\documentclass[a4paper, 11pt]{amsart}
\usepackage{amsthm}
\usepackage{amsfonts,amssymb,amsmath,amscd}
\usepackage[margin=2.9 cm]{geometry}

\usepackage{ytableau}
\usepackage{nicematrix}
\usepackage{tikz-cd}
\usepackage{xcolor}
\usepackage{enumerate}
\usepackage{multirow}
\usepackage{float}
\usepackage{mathtools}
\usepackage{mathrsfs}
\usepackage{bm}
\usepackage[pagebackref=true]{hyperref}

\usepackage{hyperref}

\definecolor{fgreen}{RGB}{44,144, 14}

\renewenvironment{proof}{{\bfseries Proof.}}{\qed}
 
\numberwithin{equation}{section}
\newtheorem{theorem}{Theorem}[section]
\newtheorem*{theorem*}{Theorem} 
\newtheorem{proposition}[theorem]{Proposition} 
\newtheorem{corollary}[theorem]{Corollary} 
\newtheorem{lemma}[theorem]{Lemma} 

\theoremstyle{definition}
\newtheorem{definition}[theorem]{Definition} 
\newtheorem{remark}[theorem]{Remark}

\everymath{\displaystyle}
\makeatletter
\@namedef{subjclassname@2020}{\textup{2020} Mathematics Subject Classification}
\makeatother

\begin{document}

\title[Reversible Higgs Bundles]{Reversible Higgs Bundles}
\author[K. Gongopadhyay]{Krishnendu Gongopadhyay}
\author[T. Nayak]{Tathagata Nayak}
\address{Indian Institute of Science Education and Research (IISER) Mohali,
	Knowledge City,  Sector 81, S.A.S. Nagar 140306, Punjab, India}
\email{krishnendu@iisermohali.ac.in}

\address{Indian Institute of Science Education and Research (IISER) Mohali,
	Knowledge City,  Sector 81, S.A.S. Nagar 140306, Punjab, India}
 \email{tathagatanayak68@gmail.com}



\textit{}

\begin{abstract}
We study reversible Higgs bundles on a compact Riemann surface $\Sigma$ of genus
$g\geq 2$, that is, $G$-Higgs bundles $(E,\Phi)$ with $(E,\Phi)\cong(E,-\Phi)$.
Using the non-abelian Hodge correspondence, we show that the moduli space of
reversible Higgs bundles carries the structure of a real affine algebraic set.
We then describe reversibility through spectral data: in even rank, for a
smooth spectral curve, reversible Higgs bundles correspond to line bundles invariant under
the involution $\eta\mapsto-\eta$; in odd rank, under natural genericity assumptions for which
the spectral curve is nodal, we give a gluing criterion together with an explicit polystability
condition. Finally, for $SL(2n,\mathbb{C})$ we construct a relative Fourier--Mukai transform
over the regular locus of the Hitchin base and prove that it maps the structure sheaf of the
reversible locus to an object whose support is exactly the regular degree-zero
$PSp(2n,\mathbb{C})$-locus in the Langlands dual $PGL(2n,\mathbb{C})$-moduli space. This
support is of type $(B,B,B)$.
\end{abstract}

%\tableofcontents
	
	\maketitle

\section{Introduction}

Let $\Sigma$ be a compact connected Riemann surface of genus $g \geq 2$, with canonical
bundle $K$, and let $\Gamma = \pi_1(\Sigma)$. For a complex reductive group $G$, let
$\mathcal{M}(\Sigma,G)$ be the moduli space of polystable $G$-Higgs bundles $(E,\Phi)$ with
vanishing rational characteristic classes. The moduli space can be studied from several points
of view. By the non-abelian Hodge correspondence \cite{hi, Si88, Co88, GGM09}, it is in
bijection with the character variety $\mathcal{X}_G(\Gamma)$. Its smooth locus carries a
hyper-K\"ahler metric with complex structures $J_1, J_2, J_3$ \cite{hi}. Through the Hitchin
fibration, its generic points are described by line bundles on spectral curves \cite{bnr}.
Finally, for a Langlands dual pair $G$ and ${}^LG$, the two Hitchin systems live over the same
base, and their generic fibres are dual abelian varieties \cite{dp}. In the work of
Kapustin and Witten \cite{kw}, this duality is expected to act on branes: a brane of type
$(B,A,A)$ for $G$ should correspond, under a fibrewise Fourier--Mukai transform, to a brane of
type $(B,B,B)$ for ${}^LG$.

In this paper we study the involution that changes the sign of the Higgs field,
\[
  (E,\Phi) \longmapsto (E,-\Phi).
\]
We call a Higgs bundle $(E,\Phi)$ \emph{reversible} if $(E,\Phi) \cong (E,-\Phi)$. This notion is motivated by the concept of reversibility in groups, see \cite{OS} for an exposition of this topic, and the idea of adjoint reality in Lie algebras, cf. \cite{GM}, \cite{GM2},  \cite{GML}. 
We describe reversible Higgs bundles in terms of surface group representations and in terms of spectral data. On the smooth locus, the reversible locus is a $(B,A,A)$-brane, whose mirror is predicted by physics to be a $(B,B,B)$-brane \cite{kw, bgh}. Our main result determines the support of this mirror for $SL(2n,\mathbb C)$ by an explicit Fourier–Mukai transform. 

The involution $(E,\Phi)\mapsto(E,-\Phi)$ is closely related
to Simpson's cyclotomic harmonic bundles \cite[\S\S 6.1--6.2]{Si09},
the case $m=2$ corresponding to the sign-change symmetry.
Simpson also describes the associated cyclic eigenspace
decomposition and, in the maximal parabolic case, the structure
in terms of parabolic line bundles and cyclic Higgs maps.
The involution $(E,\Phi)\mapsto(E,-\Phi)$ and its fixed loci were studied by
Garc\'ia-Prada \cite{ga} and Garc\'ia-Prada--Ramanan \cite{gr}. For a complex
semisimple group, \cite[Proposition~8.5]{gr} shows that, under the non-abelian
Hodge correspondence, this involution corresponds to composing representations
with the conjugation defining the compact real form (see \cite{ga} for
$SL(n,\mathbb{C})$), and \cite[Theorem~8.10]{gr} shows that its fixed points in
the smooth locus are $(J_1,\Omega_1)$-complex Lagrangian, that is, they support
$(B,A,A)$-branes. In a different setting, Hausel and Hitchin
\cite[Theorem~1.5]{hh22} compute the relative Fourier--Mukai
transform of the structure sheaf of a Lagrangian associated
with very stable Higgs bundles, obtaining a hyperholomorphic
bundle. Our contribution is to describe the reversible locus
spectrally, including an explicit gluing criterion in odd rank,
and to identify globally over $\mathcal{A}_{\mathrm{reg}}$
the support of the relative Fourier--Mukai transform of its
structure sheaf with
$e\bigl(\mathcal{M}^0_{\mathrm{reg}}(PSp(2n,\mathbb{C}))\bigr)$.

\subsection*{ Main Result}
We work with the Langlands dual pair $SL(2n,\mathbb{C})$ and $PGL(2n,\mathbb{C})$. Let
$\mathcal{A}$ be their common Hitchin base, let $\mathcal{A}_{\mathrm{reg}} \subset \mathcal{A}$
be the locus of smooth spectral curves, and put
\[
  X = \mathcal{M}_{\mathrm{reg}}(SL(2n,\mathbb{C})), \qquad
  \hat X = \mathcal{M}^0_{\mathrm{reg}}(PGL(2n,\mathbb{C})),
\]
where the superscript $0$ denotes the topologically trivial component. Fixing a theta
characteristic identifies $X$ with the relative Prym variety of the universal spectral curve
over $\mathcal{A}_{\mathrm{reg}}$, and $\hat X$ is canonically identified with the dual abelian
scheme. The normalized relative Poincar\'e line bundle on
$X \times_{\mathcal{A}_{\mathrm{reg}}} \hat X$ defines a relative Fourier--Mukai transform
\[
  \Psi \colon D^b\big(\mathrm{Coh}(X)\big) \longrightarrow D^b\big(\mathrm{Coh}(\hat X)\big).
\]
Let $\mathcal{M}^{\mathrm{rev}}_{\mathrm{reg}} \subset X$ be the reversible locus, and let
$\mathcal{E}_{\mathrm{global}}$ be its structure sheaf, pushed forward to $X$.

\begin{theorem*}[Theorem~\ref{thm:main}]
The support of $\Psi(\mathcal{E}_{\mathrm{global}})$ is
\[
  \operatorname{Supp}\Psi(\mathcal{E}_{\mathrm{global}})
  = e\big(\mathcal{M}^0_{\mathrm{reg}}(PSp(2n,\mathbb{C}))\big),
\]
where $e \colon \mathcal{M}^0_{\mathrm{reg}}(PSp(2n,\mathbb{C})) \hookrightarrow \hat X$ is the
closed embedding given by extension of structure group along
$PSp(2n,\mathbb{C}) \subset PGL(2n,\mathbb{C})$. This support is a complex submanifold of
$\hat X$ with respect to $J_1$, $J_2$ and $J_3$, which is of type $(B,B,B)$.
\end{theorem*}

\subsection*{Idea of the proof}

The proof has a fibrewise part and a global part. Fix $a$ in the even regular locus, let
$S_a$ be the spectral curve, $C = S_a/\sigma$ its quotient by the involution, and
$A = \operatorname{Prym}(S_a/\Sigma)$. After translating by the Hitchin section, the reversible
part of the fibre becomes the fixed subgroup $B$ of $\sigma^*$ in $A$. Its identity component is
the abelian subvariety $B_0 = p^*\operatorname{Prym}(C/\Sigma)$, where $p \colon S_a \to C$ is the
quotient map, and $B$ is a finite union of translates of $B_0$. By results of Mukai \cite{Mu81},
the Fourier--Mukai transform of the structure sheaf of $B$ is supported on the annihilator
$B_0^\perp \subset A^\vee$, and the norm maps identify $B_0^\perp$ with the
$PSp(2n,\mathbb{C})$-fibre (Proposition~\ref{prop:fibrewise}).

To pass from fibres to the whole family, the transform must be restricted to fibres. The sheaf
$\mathcal{E}_{\mathrm{global}}$ is supported over the even locus, so it is not flat over
$\mathcal{A}_{\mathrm{reg}}$, and its derived restriction to a fibre involves a derived
self-intersection. We avoid this by restricting the construction to the even base
$T = \mathcal{A}^{\mathrm{even}}_{\mathrm{reg}}$: flat base change gives
$\Psi(\iota_*\mathcal{G}) \cong \hat\iota_*\Psi_T(\mathcal{G})$ (Lemma~\ref{lem:evenbase}), and
over $T$ the structure sheaf of the reversible locus is flat. Derived base change then reduces
the support computation to the fibrewise one (Proposition~\ref{prop:support}). Finally, the
extension map $e$ is a closed embedding whose fibres are identified with $B_0^\perp$
(Proposition~\ref{prop:psp}), and it is holomorphic for all three complex structures.

\subsection*{Relation to Hitchin's proposal}

A mirror of this kind was proposed by Hitchin \cite{Hi13}, in the closely related setting of
$U(m,m)$-Higgs bundles inside the $GL(2m,\mathbb{C})$-moduli space. There the spectral curve
$S$ has the form $\eta^{2m} + a_2\eta^{2m-2} + \cdots + a_{2m} = 0$, with involution
$\sigma(\eta) = -\eta$, and the generic fibres of the brane are unions of translates of
$\operatorname{Pic}^0(S/\sigma)$. Hitchin observed that dualizing the inclusion
$\operatorname{Pic}^0(S/\sigma) \subset \operatorname{Pic}^0(S)$ gives the Prym variety of
$S \to S/\sigma$, which is the generic fibre of the $Sp(2m,\mathbb{C})$-Hitchin system. On this
basis he proposed the $Sp(2m,\mathbb{C})$-moduli space as the support of the mirror brane,
together with a candidate hyperholomorphic bundle on it. This identification was obtained
fibre by fibre from the SYZ picture, and it was formulated as a conjecture: no Fourier--Mukai
transform of the brane was computed.

Independently, our Theorem~\ref{thm:main} proves the support part of this picture, globally over the regular
locus, for $SL(2n,\mathbb{C})$ and $PGL(2n,\mathbb{C})$. Compared with \cite{Hi13}, the new
features are the following.
\begin{enumerate}
  \item The dual support is obtained as the support of an actual object,
    $\Psi(\mathcal{E}_{\mathrm{global}}) \in D^b(\mathrm{Coh}(\hat X))$, computed by a
    Fourier--Mukai transform, rather than from a fibrewise duality of abelian varieties.
  \item The transform is global: it is defined over the whole regular locus
    $\mathcal{A}_{\mathrm{reg}}$, not only over the even locus
    $\mathcal{A}^{\mathrm{even}}_{\mathrm{reg}}$ where reversible Higgs bundles live. In
    particular, the support is shown to be empty over the rest of the base.
  \item The result is stated and proved for the Langlands dual pair $SL(2n,\mathbb{C})$ and
    $PGL(2n,\mathbb{C})$. On each even regular fibre the support is
    $\operatorname{Prym}(S_a/C)/\pi_a^*\operatorname{Jac}(\Sigma)[2]$, which is the fibre of the
    degree-zero $PSp(2n,\mathbb{C})$-Hitchin system.
  \item The support does not depend on the choice of theta characteristic.
\end{enumerate}

\subsection*{Further results}
 Before the main theorem, we establish the
following results on reversible Higgs bundles. 
For completeness, we present the representation-theoretic
and $(B,A,A)$-brane properties of reversible Higgs bundles
in our setting, indicating their precise relation to
the results of \cite{gr}.

\emph{Representations (\S\ref{real}).} For any complex reductive group $G$, Mostow's
theorem \cite{Mo55} allows us to realize $G$ as a self-adjoint subgroup of
$GL(n,\mathbb{C})$, so that $\theta(A)=\big(\overline{A}^{\,t}\big)^{-1}$ is a
Cartan involution of $G$. If $(E,\Phi)$ corresponds to the representation $\rho$,
then $(E,-\Phi)$ corresponds to $\theta\circ\rho$ (Theorem~\ref{thm:reversal}).
For semisimple $G$ this is the special case of \cite[Proposition~8.5(2)]{gr} in
which the involution considered there is $(E,\Phi)\mapsto(E,-\Phi)$; our proof is a direct parallel-transport argument, and it
applies to any reductive $G$. As a consequence, the reversible moduli space is
identified with the fixed locus of an anti-regular involution of
$\mathcal{X}_G(\Gamma)$, and therefore carries the structure of a real affine
algebraic set (Theorem~\ref{thm:real-affine}).

\emph{The $(B,A,A)$-brane (\S\ref{revbrane}).} For a connected semisimple group $G$ and each
topological type $c\in\pi_1(G)$, the reversible locus in the smooth moduli space is
complex with respect to $J_1$ and Lagrangian with respect to $\omega_2$ and
$\omega_3$ (Theorem~\ref{thm:BAA}). This is a special case of
\cite[Theorem~8.10(2)]{gr}; we give a complete proof in the gauge-theoretic
setting of \cite{bg}.

\emph{Spectral data (\S\ref{sec:hitchin}).} For $GL(n,\mathbb{C})$ with $n$ even and smooth spectral curve,
reversible Higgs bundles correspond to line bundles $L$ on the spectral curve with
$L \cong \sigma^*L$ (Theorem~\ref{thm:even-fibre}), and the identity component of the invariant
subgroup of $\operatorname{Pic}^0(S_a)$ is $p^*\operatorname{Pic}^0(C)$
(Corollary~\ref{cor:identity-component}). For $n = 2m+1$ odd, the spectral curve is the nodal
curve $Z \cup S'$, where $Z$ is the zero section. When $S'$ is smooth and the nodes are
transverse, reversible Higgs bundles with locally free spectral sheaf correspond to triples
$(L, L', \{\phi_{x_i}\})$, consisting of a line bundle on $Z$, a line bundle on $S'$ and gluing
isomorphisms at the nodes, subject to two conditions: the polystability condition
$\lvert 2m\deg L - \deg L' \rvert < 2m(2m+1)(g-1)$, and the existence of an isomorphism
$\alpha \colon L' \to \sigma^*L'$ with $\alpha(x_i) = 1$ at every node
(Theorem~\ref{thm:odd-fibre}). The second condition is stronger than $L' \cong \sigma^*L'$.

\subsection*{Organization}

\S\ref{real} proves the representation-theoretic description of reversibility and the real
structure of the reversible moduli space. \S\ref{revbrane} proves the $(B,A,A)$ property on the smooth
locus. \S\ref{sec:hitchin} describes reversibility in terms of spectral data in even and odd rank.
\S\ref{dual} constructs the relative Fourier--Mukai transform, computes the support of the
transformed structure sheaf, and identifies it with the regular degree-zero
$PSp(2n,\mathbb{C})$-locus.

\section{Reversibility and surface group representations}\label{real}

Throughout, $\Sigma$ denotes a compact connected Riemann surface of genus
$g \geq 2$, with canonical bundle $K$, and $\Gamma = \pi_1(\Sigma, x)$ for a
fixed base point $x \in \Sigma$. For a matrix $A$ we write $\overline{A}^{\,t}$
for its conjugate transpose.

Let $G \subset GL(n,\mathbb{C})$ be a complex reductive algebraic group. By a
theorem of Mostow \cite{Mo55}, since $G$ is reductive, there exists
$P \in GL(n,\mathbb{C})$ such that $PGP^{-1}$ is invariant under
$A \mapsto \overline{A}^{\,t}$. Replacing $G$ by the isomorphic group
$PGP^{-1}$, we assume without loss of generality that $G$ is
\emph{self-adjoint}, i.e.
\[
  A \in G \implies \overline{A}^{\,t} \in G .
\]
Consequently, the Lie algebra $\mathfrak{g} \subset \mathfrak{gl}(n,\mathbb{C})$
is closed under $X \mapsto \overline{X}^{\,t}$, the map
\[
  \theta \colon G \longrightarrow G, \qquad \theta(A) = \big(\overline{A}^{\,t}\big)^{-1},
\]
is a Cartan involution of $G$, and its fixed point set
$H := G \cap U(n)$ is a maximal compact subgroup of $G$.

We say that a $G$-Higgs bundle $(E,\Phi)$ has vanishing rational characteristic classes if, for every character $\chi: G\to\mathbb C^*$, the line bundle $E\times_\chi\mathbb C$ has degree $0$. For $G = GL(n,\mathbb C)$ this means $\deg E = 0$; for semisimple $G$ it is satisfied by every $G$-Higgs bundle.

\begin{theorem}\label{thm:reversal}
Let $G \subset GL(n,\mathbb{C})$ be as above, and let $(E,\Phi)$ be a
polystable $G$-Higgs bundle with vanishing rational characteristic classes over $\Sigma$, with corresponding
completely reducible representation $\rho \colon \pi_1(\Sigma) \to G$. Then the
representation corresponding to the Higgs bundle $(E,-\Phi)$ is
$\big(\overline{\rho}^{\,t}\big)^{-1}$.
\end{theorem}

\begin{proof}
Let $\mathbb E:= E \times_G \mathbb{C}^n$ be the holomorphic vector bundle associated
to the principal $G$-bundle $E$ via the inclusion $G \subset GL(n,\mathbb{C})$;
equivalently, $\mathbb E$ is obtained by first extending the structure group of $E$ to
$GL(n,\mathbb{C})$ and then taking the associated rank $n$ vector bundle. Via
$\mathfrak{g} \subset \mathfrak{gl}(n,\mathbb{C})$, the adjoint bundle
$E(\mathfrak{g})$ is a holomorphic subbundle of $\operatorname{End}(\mathbb E)$, so that
\[
  \Phi \in H^0\big(\Sigma, E(\mathfrak{g}) \otimes K\big) \subset H^0\big(\Sigma, \operatorname{End}(\mathbb E) \otimes K\big).
\]

By the Hitchin--Simpson theorem \cite{hi, Si88} (see \cite{GGM09} for
general reductive $G$), since $(E,\Phi)$ is polystable with vanishing rational characteristic classes, there is
a reduction of structure group of $E$ to the maximal compact subgroup $H$ such
that the Hitchin equation
\[
  F(A) + [\Phi, \Phi^{*h}] = 0
\]
holds, where $\nabla_A$ is the Chern connection, i.e. the unique $H$-connection
compatible with the holomorphic structure. Since $H \subset U(n)$ preserves the
standard hermitian form on $\mathbb{C}^n$, this reduction induces a hermitian
metric $h$ on $\mathbb E$; $\nabla_A$ is then the Chern connection of $(\mathbb E,h)$, and
$\Phi^{*h}$ denotes the adjoint of $\Phi$ with respect to $h$. Since $\mathfrak g$ is self-adjoint, $\Phi^{*h}\in\Omega^{0,1}(\Sigma,E(\mathfrak g))$. Thus $(E,\Phi,h)$ is a harmonic Higgs
bundle. Here $\nabla_A$ is locally written as $d + A$ for some
$A \in \Omega^1(\Sigma, \mathfrak{g})$, with curvature
$F(A) = dA + A \wedge A$.

Consider the connection
\[
  \nabla = \nabla_A + \Phi + \Phi^{*h} = \nabla_A + \Psi, \qquad \text{where } \Psi = \Phi + \Phi^{*h}.
\]
Since $\Psi$ takes values in $E(\mathfrak{g})$, $\nabla$ is a $G$-connection.
Its curvature is
\[
  F(A) + [\Phi,\Phi^{*h}] + \bar\partial_A \Phi + \partial_A \Phi^{*h},
\]
 the remaining terms being $(2,0)$ or $(0,2)$-forms, which vanish since $\Sigma$ has complex dimension one. The first
two terms cancel by the Hitchin equation, and
$\bar\partial_A\Phi = 0 = \partial_A \Phi^{*h}$ since $\Phi$ is holomorphic.
Hence $\nabla$ is flat, and its holonomy is the representation $\rho$
corresponding to $(E,\Phi)$.

Now consider the $G$-Higgs bundle $(E,-\Phi)$. It is again polystable, since a
reduction of structure group is compatible with $\Phi$ if and only if it is
compatible with $-\Phi$. Moreover, $(E,-\Phi,h)$ is also a harmonic Higgs
bundle, with the same metric $h$: indeed $-\Phi$ is holomorphic and
\[
  F(A) + [-\Phi, (-\Phi)^{*h}] = F(A) + [-\Phi, -\Phi^{*h}] = F(A) + [\Phi, \Phi^{*h}] = 0 .
\]
Therefore, by the same argument, the connection
\[
  \nabla' = \nabla_A - \Phi - \Phi^{*h} = \nabla_A - \Psi
\]
is a flat $G$-connection on $\mathbb E$, and its holonomy $\rho'$ is the representation
corresponding to $(E,-\Phi)$.

Let $\gamma \colon [0,1] \to \Sigma$ be a loop based at $x$. Suppose
$\sigma \colon [0,1] \to \mathbb E$ is a flat section of $\mathbb E$ along $\gamma$ with respect
to $\nabla$. This implies
\[
  \nabla \sigma = 0 \implies (\nabla_A + \Psi)\sigma = 0 \implies \nabla_A \sigma = -\Psi\sigma .
\]
Similarly, let $\tau$ be a flat section of $\mathbb E$ along $\gamma$ with respect to
$\nabla'$. This implies
\[
  \nabla' \tau = 0 \implies \nabla_A \tau = \Psi \tau .
\]
Here all covariant derivatives are taken along $\dot\gamma$. Since $\nabla_A$
is the Chern connection of $(\mathbb E,h)$, it is compatible with $h$, and hence
\[
  \frac{d}{dt}\, h(\sigma,\tau) = h(\nabla_A\sigma,\tau) + h(\sigma,\nabla_A\tau) = h(-\Psi\sigma,\tau) + h(\sigma,\Psi\tau).
\]
Note that $\Psi = \Phi + \Phi^{*h}$ is self-adjoint with respect to $h$, that is,
$h(\Psi\sigma,\tau) = h(\sigma,\Psi\tau)$. Therefore
$h(-\Psi\sigma,\tau) + h(\sigma,\Psi\tau) = 0$, and consequently
\begin{equation}\label{eq:1.1}
  \frac{d}{dt}\, h(\sigma,\tau) = 0 \implies h\big(\sigma(t),\tau(t)\big) \text{ is constant for } t \in [0,1].
\end{equation}

Fix a frame of $\mathbb E_x$ given by a point of the reduction to $H$ over $x$. This
frame is $h$-orthonormal, and with respect to it the holonomies
$\rho, \rho' \colon \pi_1(\Sigma,x) \to G$ of $\nabla, \nabla'$ are given by
$\sigma(1) = \rho(\gamma)\sigma(0)$ and $\tau(1) = \rho'(\gamma)\tau(0)$. Let
$\sigma(0) = v$ and $\tau(0) = w$, where $v,w \in \mathbb E_x$. From \eqref{eq:1.1} we
obtain
\[
  h\big(\sigma(0),\tau(0)\big) = h\big(\sigma(1),\tau(1)\big) \implies h(v,w) = h\big(\rho(\gamma)v, \rho'(\gamma)w\big).
\]
In the chosen orthonormal frame, $h(v,w) = \overline{v}^{\,t} w$, so
\begin{equation}\label{eq:1.2}
  \overline{v}^{\,t} w = \overline{\big(\rho(\gamma)v\big)}^{\,t}\big(\rho'(\gamma)w\big) = \overline{v}^{\,t}\Big(\overline{\rho(\gamma)}^{\,t}\rho'(\gamma)\Big) w .
\end{equation}
Moreover, for any $v,w \in \mathbb E_x$ there exist unique sections $\sigma, \tau$
along $\gamma$ with initial values $v, w$ respectively, such that
$\nabla\sigma = 0$ and $\nabla'\tau = 0$. Therefore \eqref{eq:1.2} holds for
all $v,w \in \mathbb E_x$, which implies
\begin{equation}\label{eq:1.3}
  \overline{\rho(\gamma)}^{\,t}\rho'(\gamma) = I \implies \rho'(\gamma) = \Big(\overline{\rho(\gamma)}^{\,t}\Big)^{-1} = \theta\big(\rho(\gamma)\big).
\end{equation}
Since $\gamma$ was arbitrary, \eqref{eq:1.3} holds for all
$\gamma \in \pi_1(\Sigma,x)$, i.e. $\rho' = \big(\overline{\rho}^{\,t}\big)^{-1}$.
Since $G$ is self-adjoint, $\rho'$ indeed takes values in $G$. Finally, a
different choice of frame from the reduction to $H$ changes $\rho$ and $\rho'$
by conjugation by the same element $k \in H$; as $\theta(k) = k$, the relation
$\rho' = \theta \circ \rho$ is independent of this choice.
\end{proof}

Let $G$ be an affine complex reductive algebraic group, and let
$\operatorname{Hom}(\Gamma, G)$ denote the set of homomorphisms from $\Gamma$ to
$G$. It is an affine algebraic variety, since $\Gamma$ is finitely presented.
The group $G$ acts on $\operatorname{Hom}(\Gamma,G)$ by conjugation, and the
categorical (GIT) quotient
\[
  \mathcal{X}_G(\Gamma) = \operatorname{Hom}(\Gamma,G) /\!\!/ G
\]
is called the $G$-\emph{character variety} of $\Gamma$. Its points are in
bijection with the closed $G$-orbits, which are the  conjugacy classes of
completely reducible representations; in particular,
$\mathcal{X}_G(\Gamma) = \operatorname{Hom}^{cr}(\Gamma,G)/G$, where
$\operatorname{Hom}^{cr}(\Gamma,G)$ denotes the set of completely reducible
representations. Since $G$ is reductive, $\mathcal{X}_G(\Gamma)$ is an affine
algebraic set, cf.\ \cite[\S 11]{si}.

Further, the non-abelian Hodge correspondence \cite{hi, Si88, Co88, GGM09}
gives a bijective correspondence between the moduli space $\mathcal{M}(\Sigma,G)$
of polystable $G$-Higgs bundles with vanishing rational characteristic classes  over $\Sigma$ and the
$G$-character variety of $\Gamma$:
\[
  \mathcal{M}(\Sigma,G) \;\overset{1-1}{\longleftrightarrow}\; \mathcal{X}_G(\Gamma).
\]
Now suppose that $\rho \colon \Gamma \to G$ is a completely reducible
representation. Then, by the non-abelian Hodge correspondence and
Theorem~\ref{thm:reversal}, $\big(\overline{\rho}^{\,t}\big)^{-1}$ is also
completely reducible, being the representation corresponding to the polystable
Higgs bundle $(E,-\Phi)$.

\begin{definition}
A $G$-Higgs bundle $(E,\Phi)$ is called \emph{reversible} if
$(E,\Phi) \cong (E,-\Phi)$.
\end{definition}

\begin{theorem}\label{thm:real-affine}
The moduli space of  polystable reversible $G$-Higgs bundles with vanishing rational characteristic classes can be
given the structure of a real affine algebraic set.
\end{theorem}

\begin{proof}
Since $G$ is self-adjoint, the map
\[
  \tilde\imath \colon \operatorname{Hom}(\Gamma,G) \longrightarrow \operatorname{Hom}(\Gamma,G),
  \qquad \rho \longmapsto \theta\circ\rho = \big(\overline{\rho}^{\,t}\big)^{-1},
\]
is well defined. It satisfies $\tilde\imath\circ\tilde\imath=\mathrm{id}$ and
$\tilde\imath(g\rho g^{-1}) = \theta(g)\,\tilde\imath(\rho)\,\theta(g)^{-1}$ with
$\theta(g)\in G$; hence it maps $G$-orbits onto $G$-orbits and, being a homeomorphism,
closed orbits onto closed orbits. Since $\mathcal{X}_G(\Gamma)=\operatorname{Hom}^{cr}(\Gamma,G)/G$
and completely reducible representations are exactly those with closed orbits,
$\tilde\imath$ induces an involution
\[
  i \colon \mathcal{X}_G(\Gamma) \longrightarrow \mathcal{X}_G(\Gamma),
  \qquad [\rho] \longmapsto \Big[\big(\overline{\rho}^{\,t}\big)^{-1}\Big].
\]

The map $\tilde\imath$ is given by complex conjugation, transposition and inversion of
matrices. Hence, for every $G$-invariant polynomial function $f$ on
$\operatorname{Hom}(\Gamma,G)$, the function $\overline{f\circ\tilde\imath}$ is again a
$G$-invariant polynomial function. Since the coordinate ring of $\mathcal{X}_G(\Gamma)$ is
the ring of $G$-invariant polynomial functions on $\operatorname{Hom}(\Gamma,G)$, it follows
that, regarding $\mathcal{X}_G(\Gamma)$ as a real affine variety via the real and imaginary
parts of its coordinates, $i$ is a morphism of real affine varieties. Therefore its fixed
point set
\[
  \operatorname{Fix}(i)
  = \Big\{ [\rho] \in \mathcal{X}_G(\Gamma) :
    \big(\overline{\rho}^{\,t}\big)^{-1} = g\rho g^{-1} \text{ for some } g \in G \Big\},
\]
which is the preimage of the diagonal under the morphism $(\mathrm{id}, i)$, is a Zariski
closed subset of $\mathcal{X}_G(\Gamma)$, and hence a real affine algebraic set.

Further, let $[(E,\Phi)] \in \mathcal{M}(\Sigma,G)$ correspond to
$[\rho] \in \mathcal{X}_G(\Gamma)$ under the non-abelian Hodge correspondence.
By Theorem~\ref{thm:reversal}, $[(E,-\Phi)]$ corresponds to $i([\rho])$. Since
the correspondence is bijective, $(E,\Phi) \cong (E,-\Phi)$ if and only if
$i([\rho]) = [\rho]$. We thus obtain a bijective correspondence
\[
  \operatorname{Fix}(i) \;\overset{1-1}{\longleftrightarrow}\;
  \big\{[(E,\Phi)] \in \mathcal{M}(\Sigma,G) : (E,\Phi) \cong (E,-\Phi)\big\}.
\]
Using this bijection, we transport the real affine algebraic set structure of
$\operatorname{Fix}(i)$ to the moduli space of reversible $G$-Higgs bundles.
\end{proof}

\section{The reversible \texorpdfstring{$(B,A,A)$}{(B,A,A)}-brane}\label{revbrane}

Throughout this section, $G$ denotes a complex connected semisimple Lie group
with Lie algebra $\mathfrak{g}$, and $H \subset G$ a maximal compact subgroup
with Lie algebra $\mathfrak{h}$.

\medskip
\noindent\textbf{Topological types.}
The topological types of principal $G$-bundles on $\Sigma$ are classified by
the elements $c \in \pi_1(G)$. We denote by $\mathcal{M}^c(G)$ the subvariety of
the moduli space $\mathcal{M}(G) = \mathcal{M}(\Sigma, G)$ of $G$-Higgs bundles
consisting of those $G$-Higgs bundles whose underlying $G$-bundle has
topological invariant $c \in \pi_1(G)$; each $\mathcal{M}^c(G)$ is connected
\cite{Li93}. Moreover, if $h$ denotes the Hitchin map (see
\S\ref{sec:hitchin} below), then for $a$ in the complement of the
discriminant locus, the connected components of the fibre $h^{-1}(a)$ are in
bijection with those of $\mathcal{M}(G)$, and hence are also classified by
$\pi_1(G)$. If we restrict $h$ to $\mathcal{M}^c(G)$, both the fibre and the
total space are connected \cite{dp}, cf.\ \cite[\S 2]{bgh}.

We denote by $\mathcal{M}^c_*(G) \subset \mathcal{M}^c(G)$ the smooth locus,
consisting of stable $G$-Higgs bundles corresponding to irreducible solutions
of the Hitchin equations (see below), and by
\[
  \mathcal{M}^{c,\mathrm{rev}}_*(G) := \big\{ [(E,\Phi)] \in \mathcal{M}^c_*(G) : (E,\Phi) \cong (E,-\Phi) \big\}
\]
the moduli space of  reversible $G$-Higgs bundles in the smooth locus of topological type
$c$. When $G$ is simply connected (for instance $G = SL(n,\mathbb{C})$), only
$c = 0$ occurs, and we simply write $\mathcal{M}^{\mathrm{rev}}$.

\medskip
\noindent\textbf{The gauge-theoretic set-up \cite[\S 2.4]{bg}.}
Let $E_G$ be a smooth principal $G$-bundle over $\Sigma$ of topological type
$c$, and let $E_H$ be a fixed reduction of $E_G$ to $H$. The set $\mathcal{A}$
of $H$-connections on $E_H$ is an affine space modelled on
$\Omega^1(\Sigma, E_H(\mathfrak{h}))$. Via the Chern correspondence,
$\mathcal{A}$ is in one-to-one correspondence with the set $\mathcal{C}$ of
holomorphic structures on $E_G$ \cite[pp.~191--192, Prop.~5]{at}, which is an
affine space modelled on $\Omega^{0,1}(\Sigma, E_G(\mathfrak{g}))$. Let
$\Omega := \Omega^{1,0}(\Sigma, E_G(\mathfrak{g}))$ and consider
\[
  \mathcal{X} := \mathcal{A} \times \Omega .
\]
Its tangent space at every point is canonically identified with
$\Omega^{0,1}(\Sigma, E_G(\mathfrak{g})) \oplus \Omega^{1,0}(\Sigma, E_G(\mathfrak{g}))$.
Let
\[
  \tau \colon \Omega^{1,0}(\Sigma, E_G(\mathfrak{g})) \longrightarrow \Omega^{0,1}(\Sigma, E_G(\mathfrak{g}))
\]
be the conjugate-linear isomorphism induced by the compact conjugation of
$\mathfrak{g}$ with respect to $\mathfrak{h}$, combined with conjugation of
forms; we use the same symbol for its inverse. Via the identification
$\mathcal{A} \cong \mathcal{C}$, for
$\alpha \in \Omega^{0,1}(\Sigma, E_G(\mathfrak{g}))$ and
$\beta \in \Omega^{1,0}(\Sigma, E_G(\mathfrak{g}))$ we have the following three
complex structures on $\mathcal{X}$:
\begin{align*}
  J_1(\alpha,\beta) &= \big(\sqrt{-1}\,\alpha,\ \sqrt{-1}\,\beta\big),\\
  J_2(\alpha,\beta) &= \big(-\sqrt{-1}\,\tau(\beta),\ \sqrt{-1}\,\tau(\alpha)\big),\\
  J_3(\alpha,\beta) &= \big(\tau(\beta),\ -\tau(\alpha)\big).
\end{align*}
Consider also the Riemannian metric $\mathrm{g}$ on $\mathcal{X}$ defined, for
$\alpha,\alpha' \in \Omega^{0,1}(\Sigma, E_G(\mathfrak{g}))$ and
$\beta,\beta' \in \Omega^{1,0}(\Sigma, E_G(\mathfrak{g}))$, by
\[
  \mathrm{g}\big((\alpha,\beta),(\alpha',\beta')\big)
  = \operatorname{Re}\left[\, -2\sqrt{-1}\int_\Sigma
      \Big( B\big(\tau(\alpha),\alpha'\big) + B\big(\beta,\tau(\beta')\big) \Big) \right],
\]
where $B$ is the Killing form. The expression in brackets is a Hermitian form on the tangent space, positive definite on the diagonal; $\mathrm g$ is its real part. The complex structures $J_1, J_2, J_3$ satisfy
the quaternion relations (in particular $J_3 = J_1 J_2$), and together with
$\mathrm{g}$ they define a hyper-K\"ahler structure on $\mathcal{X}$, with
K\"ahler forms $\omega_i(\cdot,\cdot) = \mathrm{g}(J_i\,\cdot,\cdot)$,
$i = 1,2,3$.

The gauge group $\mathcal{H}$ of $E_H$ acts on $\mathcal{X}$ preserving the
hyper-K\"ahler structure, with moment maps \cite{hi}
\[
  \mu_1(A,\varphi) = F_A - [\varphi, \tau(\varphi)], \qquad
  \mu_2(A,\varphi) = \operatorname{Re}\big(\bar\partial_A\varphi\big), \qquad
  \mu_3(A,\varphi) = \operatorname{Im}\big(\bar\partial_A\varphi\big).
\]
The quotient $\mu^{-1}(0)/\mathcal{H}$, where $\mu = (\mu_1,\mu_2,\mu_3)$, is
the moduli space of solutions of the Hitchin equations. Restricting to the
subset $\mu^{-1}_*(0) \subset \mu^{-1}(0)$ of irreducible solutions, the
hyper-K\"ahler quotient
\[
  \mathcal{M}_* := \mu^{-1}_*(0)/\mathcal{H}
\]
is a hyper-K\"ahler manifold which, by \cite[Theorem~2.3]{bg}, is
homeomorphic to the smooth locus $\mathcal{M}^c_*(G)$, via
$[(A,\varphi)] \mapsto [(E_G, \bar\partial_A, \varphi)]$. The Riemannian
metric $\mathrm{g}$ and the complex structures $J_1, J_2, J_3$ on
$\mathcal{M}_*$ are precisely those induced from the corresponding structures
on $\mathcal{X}$ by the quotient construction: the tangent space of $\mathcal M_*$ at $[(A,\varphi)]$ is identified with the $\mathrm g$-orthogonal complement, in $T_{(A,\varphi)}\mu^{-1}_*(0)$, of the tangent space to the $\mathcal H$-orbit through $(A,\varphi)$.

\medskip
\noindent\textbf{A standard fact.}
Let $(M, \mathrm{g}, I)$ be a K\"ahler manifold with K\"ahler form
$\omega = \mathrm{g}(I\cdot,\cdot)$, and let $\sigma \colon M \to M$ be an
anti-holomorphic isometric involution. Then $\operatorname{Fix}(\sigma)$ is a
(real) Lagrangian submanifold of $(M,\omega)$. Indeed, $\operatorname{Fix}(\sigma)$
is a submanifold since $\sigma$ is an isometric involution, and at $p \in
\operatorname{Fix}(\sigma)$ its tangent space is the $(+1)$-eigenspace $V_+$ of
$d\sigma_p$. Since $d\sigma_p \circ I = -I \circ d\sigma_p$, the complex
structure $I$ interchanges $V_+$ and the $(-1)$-eigenspace $V_-$, so
$\dim_{\mathbb{R}} V_+ = \frac{1}{2}\dim_{\mathbb{R}} M$. For $u,v \in V_+$,
\[
  \omega(u,v) = \mathrm{g}(Iu,v) = \mathrm{g}\big(d\sigma_p(Iu), d\sigma_p(v)\big) = -\mathrm{g}(Iu,v) = -\omega(u,v),
\]
so $\omega|_{V_+} = 0$.

\medskip
Following \cite{kw}, a Lagrangian submanifold of a symplectic manifold is
called an $A$-brane and a complex submanifold a $B$-brane. A submanifold of a
hyper-K\"ahler manifold may be of type $A$ or $B$ with respect to each of the
complex or symplectic structures, and thus, choosing a triple of structures,
one may speak of branes of type $(B,B,B)$, $(B,A,A)$, $(A,B,A)$ and $(A,A,B)$.

The following is a special case of
\cite[Theorem~8.10(2)]{gr};
we include a direct gauge-theoretic proof.

\begin{theorem}\label{thm:BAA}
Let $G$ be a complex connected semisimple Lie group and $c \in \pi_1(G)$. The
moduli space $\mathcal{M}^{c,\mathrm{rev}}_*(G)$ of reversible
$G$-Higgs bundles in the smooth locus of topological type $c$ defines a canonical
$(B,A,A)$-brane in $\mathcal{M}^c_*(G)$ in the sense of Kapustin--Witten.
\end{theorem}

\begin{proof}
Consider the involution
\[
  \sigma \colon \mathcal{X} \longrightarrow \mathcal{X}, \qquad \sigma(\alpha,\beta) = (\alpha,-\beta),
\]
that is, $(\bar\partial_A, \varphi) \mapsto (\bar\partial_A, -\varphi)$. Since
$\sigma$ is an affine map of the affine space $\mathcal{X}$, its differential
at every point is the linear map $d\sigma(\alpha,\beta) = (\alpha,-\beta)$ on
the tangent space.

The involution $\sigma$ is holomorphic with respect to $J_1$, because
\[
  d\sigma \circ J_1(\alpha,\beta) = \sigma\big(\sqrt{-1}\,\alpha, \sqrt{-1}\,\beta\big) = \big(\sqrt{-1}\,\alpha, -\sqrt{-1}\,\beta\big) = J_1 \circ d\sigma(\alpha,\beta).
\]
Moreover, $\sigma$ is anti-holomorphic with respect to $J_2$, since
\begin{align*}
  d\sigma \circ J_2(\alpha,\beta) &= \sigma\big(-\sqrt{-1}\,\tau(\beta), \sqrt{-1}\,\tau(\alpha)\big) = \big(-\sqrt{-1}\,\tau(\beta), -\sqrt{-1}\,\tau(\alpha)\big),\\
  J_2 \circ d\sigma(\alpha,\beta) &= J_2(\alpha,-\beta) = \big(-\sqrt{-1}\,\tau(-\beta), \sqrt{-1}\,\tau(\alpha)\big) = \big(\sqrt{-1}\,\tau(\beta), \sqrt{-1}\,\tau(\alpha)\big),
\end{align*}
which implies $d\sigma \circ J_2 = -J_2 \circ d\sigma$. Similarly,
\[
  d\sigma \circ J_3(\alpha,\beta) = \big(\tau(\beta), \tau(\alpha)\big) = -J_3(\alpha,-\beta) = -J_3 \circ d\sigma(\alpha,\beta),
\]
so $\sigma$ is anti-holomorphic with respect to $J_3$. The involution $\sigma$
is also isometric:
\begin{align*}
  \mathrm{g}\big(d\sigma(\alpha,\beta), d\sigma(\alpha',\beta')\big)
  &= \operatorname{Re}\left[ -2\sqrt{-1}\int_\Sigma \Big( B\big(\tau(\alpha),\alpha'\big)
     + B\big(-\beta,\tau(-\beta')\big) \Big) \right]\\
  &= \operatorname{Re}\left[ -2\sqrt{-1}\int_\Sigma \Big( B\big(\tau(\alpha),\alpha'\big)
     + B\big(\beta,\tau(\beta')\big) \Big) \right]
   = \mathrm{g}\big((\alpha,\beta),(\alpha',\beta')\big).
\end{align*}

Since $\tau$ is real-linear, $\tau(-\varphi) = -\tau(\varphi)$, hence $\mu_1(A,-\varphi) = \mu_1(A,\varphi)$. Moreover, since $\bar\partial_A$ is linear, $\mu_2(A,-\varphi) = -\mu_2(A,\varphi)$ and $\mu_3(A,-\varphi) = -\mu_3(A,\varphi)$.
Therefore $\sigma$ preserves $\mu^{-1}(0)$. It also preserves the subset
$\mu^{-1}_*(0)$ of irreducible solutions, since a reduction of structure group
is compatible with $\varphi$ if and only if it is compatible with $-\varphi$.
Moreover, $\sigma$ is $\mathcal{H}$-equivariant, since the gauge action
$(A,\varphi) \mapsto (k\cdot A, \operatorname{Ad}_k\varphi)$ commutes with
$\varphi \mapsto -\varphi$. Therefore $\sigma$ induces an involution
\[
  \tilde\sigma \colon \mathcal{M}_*  \longrightarrow \mathcal{M}_* .
\]
Since $\sigma$ is $\mathcal{H}$-equivariant and isometric, $d\sigma$ preserves
the tangent spaces of the $\mathcal{H}$-orbits and hence their
$\mathrm{g}$-orthogonal complements. Consequently, the properties above descend
to the quotient: $\tilde\sigma$ is an isometry of $\mathcal{M}_*$ which is
holomorphic with respect to $J_1$ and anti-holomorphic with respect to $J_2$
and $J_3$.

Since $\tilde\sigma$ is a holomorphic involution with respect to $J_1$,
$\operatorname{Fix}(\tilde\sigma)$ is a complex submanifold of $\mathcal{M}_*$
with respect to $J_1$ (by Cartan's linearization theorem, near a fixed point
$\tilde\sigma$ is linear in suitable holomorphic coordinates). Further, since
$\tilde\sigma$ is an isometric involution which is anti-holomorphic with
respect to $J_2$, the standard fact above shows that
$\operatorname{Fix}(\tilde\sigma)$ is a real Lagrangian submanifold of
$\mathcal{M}_*$ with respect to the K\"ahler structure $(\mathrm{g},
\omega_2, J_2)$. Similarly, since $\tilde\sigma$ is anti-holomorphic with
respect to $J_3$, $\operatorname{Fix}(\tilde\sigma)$ is a real Lagrangian
submanifold with respect to $(\mathrm{g}, \omega_3, J_3)$.

Finally, under the homeomorphism $\mathcal{M}_* \cong \mathcal{M}^c_*(G)$ of
\cite[Theorem~2.3]{bg}, the involution $\tilde\sigma$ corresponds to
$[(E,\Phi)] \mapsto [(E,-\Phi)]$. Hence $\operatorname{Fix}(\tilde\sigma)$
corresponds to $\mathcal{M}^{c,\mathrm{rev}}_*(G)$. Therefore
$\mathcal{M}^{c,\mathrm{rev}}_*(G)$ is complex with respect to $J_1$ and
Lagrangian with respect to $\omega_2$ and $\omega_3$, that is, it defines a
$(B,A,A)$-brane.
\end{proof}

% -------------------------------------------------------------------
\section{Hitchin fibration}\label{sec:hitchin}

A natural way of studying $\mathcal{M}(G)$ is through the Hitchin fibration.
Here $G$ may be any complex connected reductive group. We denote by $p_i$,
$i = 1, \dots, k$, a homogeneous basis for the algebra of invariant polynomials
of $\mathfrak{g}$, and by $d_i$ their degrees. The Hitchin fibration is given by
\begin{equation}\label{eq:2.1}
  h \colon \mathcal{M}(G) \longrightarrow \mathcal{A}_G := \bigoplus_{i=1}^k H^0\big(\Sigma, K^{d_i}\big),
\end{equation}
\begin{equation}\label{eq:2.2}
  (E,\Phi) \longmapsto \big(p_1(\Phi), \dots, p_k(\Phi)\big).
\end{equation}
The map $h$ is referred to as the Hitchin map.

Let $\mathfrak{g}$ be one of the classical Lie algebras
$\mathfrak{gl}(n,\mathbb{C})$, $\mathfrak{sl}(n,\mathbb{C})$,
$\mathfrak{sp}(2n,\mathbb{C})$, $\mathfrak{so}(2n+1,\mathbb{C})$. Then, for the
standard representation $\mathfrak{g} \subset \mathfrak{gl}(V)$, the ring of
invariant polynomials of $\mathfrak{g}$ is generated by the functions
$X \mapsto \operatorname{tr}(X^i)$, $i \in \mathbb{N}$. Hence, for
$GL(n,\mathbb{C})$-Higgs bundles $(E,\Phi)$ of rank $n$, a homogeneous basis of
invariant polynomials can be taken to be $\operatorname{tr}(\Phi^i)$,
$1 \leq i \leq n$. Equivalently, by Newton's identities, one may use the
coefficients of the characteristic polynomial of $\Phi$, which we do below.

Let $\operatorname{Tot}(K)$ be the total space of $K$, with projection
$\pi \colon \operatorname{Tot}(K) \to \Sigma$, and denote by $\eta$ the
tautological section of $\pi^*K$ on $\operatorname{Tot}(K)$. Abusing notation,
we denote sections of the powers $K^i$ on $\Sigma$ and their pullbacks to
$\operatorname{Tot}(K)$ by the same symbols. For a Higgs bundle $(E,\Phi)$
with $h(E,\Phi) = a$, the characteristic polynomial of $\Phi$ defines a curve
$S_a \subset \operatorname{Tot}(K)$, the \emph{spectral curve} of $\Phi$, with
equation
\begin{equation}\label{eq:2.3}
  \det\big(\eta\,\mathrm{Id} - \pi^*\Phi\big) = \eta^n + a_1\eta^{n-1} + a_2\eta^{n-2} + \cdots + a_{n-1}\eta + a_n = 0,
\end{equation}
where $a_i \in H^0(\Sigma, K^i)$. For generic $a$, the curve $S_a$ is smooth,
and $\pi \colon S_a \to \Sigma$ is an $n$-sheeted branched covering. For
simplicity, we write $\det(\eta - \Phi)$ for the characteristic polynomial of
$\Phi$.

Now we consider reversible $GL(n,\mathbb{C})$-Higgs bundles of even rank $n$.
In the rest of this section, $\mathcal M(GL(n,\mathbb C))$ denotes the moduli space of polystable $GL(n,\mathbb C)$-Higgs bundles of all degrees. 
Let $(E,\Phi)$ be a reversible Higgs bundle of even rank, and suppose that
$h(E,\Phi) = a \in \bigoplus_{i=1}^n H^0(\Sigma, K^i)$. Let $S_a$ be the
corresponding spectral curve. Since $(E,\Phi) \cong (E,-\Phi)$, we have
$h(E,-\Phi) = a$. This implies
$\operatorname{tr}(\Phi^i) = \operatorname{tr}((-\Phi)^i)$ for all
$1 \leq i \leq n$. For odd $i$, $\operatorname{tr}(\Phi^i) =
\operatorname{tr}((-\Phi)^i) = -\operatorname{tr}(\Phi^i)$, so
$\operatorname{tr}(\Phi^i) = 0$. By Newton's identities, this implies $a_i = 0$
for all odd $i$. Therefore the equation of $S_a$ becomes
\[
  \eta^n + a_2\eta^{n-2} + \cdots + a_{n-2}\eta^2 + a_n = 0 .
\]
Hence $S_a$ has a natural involution $\sigma(x,\eta) = (x,-\eta)$, which
satisfies $\pi \circ \sigma = \pi$. 

We now restrict the Hitchin map $h$ to the moduli space of reversible
$GL(n,\mathbb{C})$-Higgs bundles and describe its fibres.

\begin{theorem}\label{thm:even-fibre}
Let $(E,\Phi)$ be a reversible $GL(n,\mathbb{C})$-Higgs bundle with $n$ even.
Suppose that $h(E,\Phi) = a$ and that the spectral curve $S_a$ is smooth. Then
the reversible Higgs bundles in the Hitchin fibre $h^{-1}(a)$ are in bijective
correspondence with the line bundles $L \in \operatorname{Pic}(S_a)$ such that
$L \cong \sigma^* L$.
\end{theorem}

\begin{proof}
By \cite{bnr}, there is a bijective correspondence between the (isomorphism
classes of) Higgs bundles in $h^{-1}(a)$ and the (isomorphism classes of) line
bundles on the smooth spectral curve $S_a$. Explicitly, for a line bundle $L$
on $S_a$, we obtain a Higgs bundle by considering the direct image
$E := \pi_*L$. Recall that, by definition of the direct image, for an open set
$U \subset \Sigma$ one has $H^0(\pi^{-1}(U), L) = H^0(U, \pi_*L)$.
Multiplication by the tautological section $\eta$ induces the map
\[
  H^0\big(\pi^{-1}(U), L\big) \longrightarrow H^0\big(\pi^{-1}(U), L \otimes \pi^*K\big),
\]
which, by definition of the direct image, gives
$\Phi \colon \pi_*L \to \pi_*L \otimes K$. Thus we obtain a Higgs field
$\Phi \in H^0(\Sigma, \operatorname{End}(E) \otimes K)$.

Now consider the pullback line bundle $\sigma^*L$ on $S_a$. Since
$\pi \circ \sigma = \pi$, we have $\sigma(\pi^{-1}(U)) = \pi^{-1}(U)$ for every
open set $U \subset \Sigma$, and pulling back sections along $\sigma$ gives an
isomorphism
\[
  \sigma^* \colon H^0\big(\pi^{-1}(U), L\big) \xrightarrow{\ \cong\ } H^0\big(\pi^{-1}(U), \sigma^*L\big),
\]
compatible with restrictions. Hence $\pi_*(\sigma^*L) \cong \pi_*L = E$.
Moreover, since $\sigma^*\pi^*K = \pi^*K$ and $\sigma^*\eta = -\eta$, for every local section $s$ of $L$ we have
$$\sigma^*s\otimes\eta = \sigma^*s\otimes(-\sigma^*\eta) = -\,\sigma^*(s\otimes\eta),$$
where $(\sigma^*s)(y) := s(\sigma(y))$. That is, under the isomorphism $\sigma^*: H^0(\pi^{-1}(U),L)\xrightarrow{\cong}H^0(\pi^{-1}(U),\sigma^*L)$, the map $s\mapsto s\otimes\eta$ on $\sigma^*L$ corresponds to $s\mapsto -\,s\otimes\eta$ on $L$, which pushes down to 
$-\Phi \colon \pi_*L \to \pi_*L \otimes K$. Therefore the line bundle
$\sigma^*L$ corresponds to the Higgs bundle $(E,-\Phi)$. Since the
correspondence of \cite{bnr} is bijective, $(E,\Phi) \cong (E,-\Phi)$ if and
only if $L \cong \sigma^*L$.
\end{proof}

\begin{corollary}\label{cor:identity-component}
Let $p \colon S_a \to C := S_a/\sigma$ be the quotient map. Then
$p^*\operatorname{Pic}(C)$ is contained in
$\operatorname{Pic}(S_a)^{\sigma^*} := \{L \in \operatorname{Pic}(S_a) : L \cong \sigma^*L\}$.
Moreover, $p^*\operatorname{Pic}^0(C)$ is the connected component of the
identity of the group $\operatorname{Pic}^0(S_a)^{\sigma^*}$.
\end{corollary}

\begin{proof}
The quotient $C$ is a smooth curve, being the quotient of a smooth curve by a
finite group. Since $p \circ \sigma = p$, we have
$\sigma^*p^*M \cong p^*M$ for every $M \in \operatorname{Pic}(C)$, which proves
the first assertion. 

The map $p^*$ has finite kernel on $\mathrm{Pic}^0(C)$, since $\mathrm{Nm}_{S_a/C}(p^*M) = M^{\otimes2}$ forces any $M$ in the kernel to be $2$-torsion; hence $p^*\mathrm{Pic}^0(C)$, being the image of a compact connected group under a holomorphic homomorphism, is a compact (hence closed) connected subgroup of $\mathrm{Pic}^0(S_a)^{\sigma^*}$ of dimension $g(C)$. On the other hand, the Lie algebra of $\mathrm{Pic}^0(S_a)^{\sigma^*}$ is $H^1(S_a,\mathcal O_{S_a})^\sigma\cong H^1(C,\mathcal O_C)$, of dimension $g(C)$. A closed connected subgroup of the same dimension is the identity component; hence $p^*\mathrm{Pic}^0(C)$ is the identity component.
\end{proof}

Now consider the case when $(E,\Phi)$ is a reversible Higgs bundle of odd rank
$n = 2m+1$. As above, all odd coefficients $a_i$ vanish, so the equation of the
corresponding spectral curve $S_a$ is
\[
  \eta\big(\eta^{2m} + a_2\eta^{2m-2} + \cdots + a_{2m-2}\eta^2 + a_{2m}\big) = 0 .
\]
Therefore $S_a = Z \cup S'$, where
\[
  Z : \eta = 0, \qquad S' : \eta^{2m} + a_2\eta^{2m-2} + \cdots + a_{2m-2}\eta^2 + a_{2m} = 0 .
\]
Here $Z \cong \Sigma$ is the zero section. The involution
$\sigma(x,\eta) = (x,-\eta)$ preserves both components and acts trivially on
$Z$. Note that $Z$ and $S'$ intersect at the points $(x,0)$ with
$a_{2m}(x) = 0$; if $a_{2m}$ has simple zeros, these intersections are
transversal, so $S_a$ has ordinary nodes there, and they are exactly the fixed
points of $\sigma$ on $S'$.

For $n = 1$, a Higgs bundle $(E,\phi)$ is reversible if and only if $\phi = 0$; we therefore assume $n = 2m+1$ with $m\geq1$.
\begin{theorem}\label{thm:odd-fibre}
Let $(E,\Phi)$ be a reversible $GL(n,\mathbb{C})$-Higgs bundle with
$n = 2m+1$ odd. Suppose that $h(E,\Phi) = a$, that $S'$ is smooth, and that
$a_{2m}$ has simple zeros $x_1, \dots, x_k$ (regarded as points of
$Z \cap S'$). Then the reversible Higgs bundles in $h^{-1}(a)$ whose spectral sheaf is locally free are in bijective correspondence with (isomorphism classes of) triples $(L,L',\{\phi_{x_i}\})$, where $L\in\mathrm{Pic}(Z)$, $L'\in\mathrm{Pic}(S')$ and $\phi_{x_i}: L_{x_i}\to L'_{x_i}$ are isomorphisms, such that
\[
  \big|\,2m\deg L - \deg L'\,\big| < 2m(2m+1)(g-1),
\] and there exists an isomorphism $\alpha: L'\to\sigma^*L'$ with $\alpha(x_i) = 1$ for all $i$. Here $(\sigma^*L')_{x_i}$ is identified with $L'_{x_i}$, since $\sigma(x_i) = x_i$.
\end{theorem}

\begin{proof}
A line bundle $L \in \operatorname{Pic}(Z)$, a line bundle
$L' \in \operatorname{Pic}(S')$ and isomorphisms
$\phi_{x_i} \colon L_{x_i} \to L'_{x_i}$ at the nodes define a locally free
sheaf $\mathcal{F}$ of rank $1$ on $S_a$: for an open set $V \subset S_a$,
\[
  \mathcal{F}(V) = \Big\{ (\xi,\tau) \in H^0(V \cap Z, L) \times H^0(V \cap S', L') :
  \tau(x_i) = \phi_{x_i}\big(\xi(x_i)\big) \text{ for every } x_i \in V \Big\}.
\]
In particular, $\mathcal{F}(V) = H^0(V,L')$ if $V \subset S'$,
$\mathcal{F}(V) = H^0(V,L)$ if $V \subset Z$, and
$\mathcal{F}(V) = H^0(V \cap Z, L) \times H^0(V \cap S', L')$ if $V$ contains
none of the $x_i$. Every locally free
sheaf  of rank $1$ on the nodal curve $S_a$ arises in this
way, and two triples give isomorphic line bundles if and only if they are
isomorphic as triples. By \cite{bnr} (see also \cite{Sc98} for reduced
spectral curves), $E := \pi_*\mathcal{F}$, where $\pi \colon S_a \to \Sigma$,
together with the Higgs field $\Phi \colon \pi_*\mathcal{F} \to
\pi_*\mathcal{F} \otimes K$ induced by multiplication by $\eta$,
\[
  \mathcal{F}\big(\pi^{-1}(U)\big) \xrightarrow{\ \eta\ } \mathcal{F}\big(\pi^{-1}(U)\big) \otimes H^0\big(\pi^{-1}(U), \pi^*K\big),
\]
 is a Higgs bundle with characteristic coefficients $a$, and this gives a bijection
between such Higgs bundles (not necessarily polystable) with locally free spectral
sheaf and line bundles on $S_a$.

We first  determine which of these lie in the Hitchin fibre $h^{-1}(a)$, that means,  are
polystable. Since $\Phi$ satisfies its characteristic polynomial, $E$ is a module
over $\pi_*\mathcal{O}_{S_a}$, and the $\Phi$-invariant subbundles of $E$ are exactly
the sheaves $\pi_*\mathcal{F}'$ for subsheaves $\mathcal{F}'\subset\mathcal{F}$ with
torsion-free quotient. Since $\mathcal{F}$ has rank one on each component and $S'$ is
connected (being a smooth spectral curve), the only proper nonzero such subsheaves are
$L(-D)$, supported on $Z$, and $L'(-D')$, supported on $S'$, where $D$ and $D'$ denote
the node divisors on $Z$ and $S'$ and $k=\deg D=4m(g-1)$. The corresponding
$\Phi$-invariant subbundles of $E$ are $L(-D)$ (via $Z\cong\Sigma$) and
$\pi'_*\big(L'(-D')\big)$, where $\pi'=\pi|_{S'}$; they have ranks $1$ and $2m$ and
degrees $\deg L-k$ and $\deg L'-k-2m(2m-1)(g-1)$, while
$\deg E=\deg L+\deg L'-(2m+1)2m(g-1)$. The two slope inequalities are equivalent to
$\big|2m\deg L-\deg L'\big|<2m(2m+1)(g-1)$, so this is exactly the condition for $E$
to be stable. Moreover, $E$ cannot be strictly polystable: a splitting of $E$ into
$\Phi$-invariant summands would split $\mathcal{F}$ into sheaves supported on $Z$ and
on $S'$, which is impossible for a line bundle on the connected curve $S_a$. Hence the
Higgs bundles in $h^{-1}(a)$ with locally free spectral sheaf correspond exactly to
the triples satisfying this inequality.

Recall that $\sigma \colon S_a \to S_a$ is defined by $\sigma(x,\eta) = (x,-\eta)$.
For any open set $U \subset \Sigma$, since $\sigma(\pi^{-1}(U)) = \pi^{-1}(U)$,
\[
  H^0\big(U, \pi_*\sigma^*\mathcal{F}\big) = H^0\big(\pi^{-1}(U), \sigma^*\mathcal{F}\big)
  \cong H^0\big(\sigma(\pi^{-1}(U)), \mathcal{F}\big) = H^0\big(\pi^{-1}(U), \mathcal{F}\big).
\]
Therefore $\pi_*(\sigma^*\mathcal{F}) \cong \pi_*\mathcal{F} = E$.

We now describe $\sigma^*\mathcal{F}$ in terms of gluing data. Since $\sigma$
is an involution, $(\sigma^*\mathcal{F})(V) \cong \mathcal{F}(\sigma(V))$. Since
$\sigma|_Z = \mathrm{id}$ and $\sigma(x_i) = x_i$ for all $i$, we obtain, for
every open set $V \subset S_a$,
\[
  (\sigma^*\mathcal{F})(V) \cong \Big\{ (\xi,\tau) \in H^0(V \cap Z, L) \times H^0(V \cap S', \sigma^*L') :
  \tau(x_i) = \phi_{x_i}\big(\xi(x_i)\big) \text{ for every } x_i \in V \Big\}.
\]
Therefore the sheaf $\sigma^*\mathcal{F}$ corresponds to the triple
\[
  \Big( L,\ \sigma^*L',\ \big\{\phi_{x_i} \colon L_{x_i} \to (\sigma^*L')_{x_i} = L'_{x_i}\big\} \Big).
\]

As in the proof of Theorem~\ref{thm:even-fibre}, since $\sigma^*\eta = -\eta$,
multiplication by $\eta$ on $\sigma^*\mathcal{F}$ corresponds to
multiplication by $-\eta$ on $\mathcal{F}$, and its direct image is $-\Phi$.
Hence $\sigma^*\mathcal{F}$ corresponds to the Higgs bundle $(E,-\Phi)$. In
particular, by the bijectivity of the correspondence,
\[
  (E,\Phi) \cong (E,-\Phi) \iff \mathcal{F} \cong \sigma^*\mathcal{F}.
\]

It remains to express the condition $\mathcal{F} \cong \sigma^*\mathcal{F}$
in terms of the triple. An isomorphism of triples
$(L, L', \{\phi_{x_i}\}) \to (L, \sigma^*L', \{\phi_{x_i}\})$ consists of an
automorphism of $L$, which is a scalar $\lambda \in \mathbb{C}^*$ since
$Z \cong \Sigma$ is compact and connected, and an isomorphism
$\alpha \colon L' \to \sigma^*L'$, such that
$\alpha(x_i) \circ \phi_{x_i} = \phi_{x_i} \circ \lambda$, which implies 
$\alpha(x_i) = \lambda$ for all $i$. 
Replacing $\alpha$ by $\lambda^{-1}\alpha$, we see that $\mathcal F\cong\sigma^*\mathcal F$ if and only if there is an isomorphism $\alpha: L'\to\sigma^*L'$ with $\alpha(x_i) = 1$ for all $i$.
\end{proof}

\section{Langlands duality}\label{dual}

In this section we consider reversible $SL(2n,\mathbb{C})$-Higgs bundles and
their Langlands dual counterpart on the $PGL(2n,\mathbb{C})$ side. We first
state the main result, then establish its ingredients in a sequence of
propositions (\S\S\ref{subsec:notation}--\ref{subsec:psp}), and finally prove
the theorem in \S\ref{subsec:proof}.

\begin{theorem}\label{thm:main}
The $(B,A,A)$-brane $\mathcal{M}^{\mathrm{rev}}_{\mathrm{reg}}(SL(2n,\mathbb{C}))$ of
reversible $SL(2n,\mathbb{C})$-Higgs bundles with smooth spectral curves gives
rise to a $(B,B,B)$-support in the moduli space of its Langlands dual
$PGL(2n,\mathbb{C})$-Higgs bundles. More precisely, under the relative
Fourier--Mukai transform $\Psi$ of \S\ref{subsec:FM}, the structure sheaf
$\mathcal{E}_{\mathrm{global}}$ of
$\mathcal{M}^{\mathrm{rev}}_{\mathrm{reg}}(SL(2n,\mathbb{C}))$ is mapped to an object
$\Psi(\mathcal{E}_{\mathrm{global}})$ of
$D^b\big(\mathrm{Coh}(\mathcal{M}^0_{\mathrm{reg}}(PGL(2n,\mathbb{C})))\big)$ whose
support is of type $(B,B,B)$. Moreover, this support is contained in
the moduli space of $PSp(2n,\mathbb{C})$-Higgs bundles: this is the image of
$\mathcal{M}^0_{\mathrm{reg}}(PSp(2n,\mathbb{C}))$ under extension of structure group.
\end{theorem}

% -------------------------------------------------------------------
\subsection{Notation and the regular locus}\label{subsec:notation}

Since $\mathfrak{sl}(2n,\mathbb{C}) = \mathfrak{pgl}(2n,\mathbb{C})$, the
Hitchin bases of $SL(2n,\mathbb{C})$ and $PGL(2n,\mathbb{C})$ coincide; we
denote both by
\[
  \mathcal{A} = \bigoplus_{i=2}^{2n} H^0\big(\Sigma, K^i\big),
\]
with coordinates $a = (a_2, \dots, a_{2n})$ given by the coefficients of the
characteristic polynomial, as in \eqref{eq:2.3}. The \emph{regular locus}
$\mathcal{A}_{\mathrm{reg}} \subset \mathcal{A}$ is the set of points whose spectral curve
$S_a$ is smooth. It is the complement of the discriminant locus, hence open
and dense in $\mathcal{A}$, and it is the same for both groups, cf.\
\cite{dp}, \cite[Theorem~2.1]{bgh}. Indeed, the spectral curve
\[
  S_a : \ \eta^{2n} + \sum_{k=2}^{2n} a_k\, \eta^{2n-k} = 0
\]
depends only on $a$, so both Hitchin systems have the same \emph{universal
spectral curve}
\[
  \mathcal{S} = \Big\{ (x,\eta,a) \in \operatorname{Tot}(K) \times \mathcal{A}_{\mathrm{reg}} : \eta^{2n} + \sum_{k=2}^{2n} a_k(x)\, \eta^{2n-k} = 0 \Big\} \longrightarrow \mathcal{A}_{\mathrm{reg}},
\]
with projection $\pi \colon \mathcal{S} \to \Sigma$, $(x,\eta,a) \mapsto x$, whose
restriction to $S_a$ is $\pi_a \colon S_a \to \Sigma$. Every Higgs bundle with
smooth spectral curve is stable, since a destabilizing subobject would force
$S_a$ to be reducible (cf.\ \cite[\S 4]{bgh}). Moreover, every point of the regular locus (Higgs bundle with smooth spectral curve)  is a smooth point of the moduli space of $SL(2n,\mathbb C)$,  $PGL(2n,\mathbb C)$ or $PSp(2n,\mathbb C)$-Higgs bundles, respectively. Indeed, a stable Higgs bundle whose automorphism group is the centre of the structure group is a smooth point of the moduli space, and in our case the automorphism groups are as follows: for $SL(2n,\mathbb C)$, an automorphism of $(\pi_*L,\Phi)$ is given by an element of $H^0(S_a,\mathcal O^*) = \mathbb C^*$ of determinant one, hence lies in the centre $\mu_{2n}$; for $PGL(2n,\mathbb C)$ and $PSp(2n,\mathbb C)$, an automorphism would give $L\otimes\pi_a^*M\cong L$ for some $M\in\mathrm{Pic}^0(\Sigma)$, forcing $M\cong\mathcal O$ by injectivity of $\pi_a^*$ (as in the proof of Proposition~\ref{prop:fibrewise}(i)), so the automorphism group is trivial.

Let
\[
  \mathcal{A}^{\mathrm{even}} = \{ a \in \mathcal{A} : a_k = 0 \text{ for all odd } k \} = \bigoplus_{i=1}^n H^0\big(\Sigma, K^{2i}\big),
  \qquad
  \mathcal{A}^{\mathrm{even}}_{\mathrm{reg}} := \mathcal{A}_{\mathrm{reg}} \cap \mathcal{A}^{\mathrm{even}}.
\]
Thus $\mathcal{A}^{\mathrm{even}}_{\mathrm{reg}}$ is open in the linear subspace
$\mathcal{A}^{\mathrm{even}}$ and closed in $\mathcal{A}_{\mathrm{reg}}$. By
\S\ref{sec:hitchin}, every reversible Higgs bundle lies over
$\mathcal{A}^{\mathrm{even}}$. For $a \in \mathcal{A}^{\mathrm{even}}_{\mathrm{reg}}$, the curve
$S_a$ carries the involution $\sigma(x,\eta) = (x,-\eta)$, and we have maps of
smooth curves
\[
  S_a \xrightarrow{\ p\ } C := S_a/\sigma \xrightarrow{\ \pi_C\ } \Sigma, \qquad \pi_a = \pi_C \circ p .
\]
Note that $a_{2n} \not\equiv 0$, since otherwise the equation of $S_a$ would be
divisible by $\eta^2$, contradicting smoothness. Moreover, the zeros of
$a_{2n}$ are simple: at a zero $x_0$ of $a_{2n}$, the point $(x_0,0)$ lies on
$S_a$, the $\eta$-derivative of the equation vanishes there, and the
derivative in the base direction is $a_{2n}'(x_0)$. Hence $\sigma$ has exactly
$\deg K^{2n} = 4n(g-1)$ fixed points, and $p$ is a double cover ramified at
these points.

As explained in \S\ref{revbrane},  $\mathcal{M}(SL(2n,\mathbb{C}))$ is connected since
$\pi_1(SL(2n,\mathbb{C})) = 0$, whereas \\$\mathcal{M}(PGL(2n,\mathbb{C}))$ has $2n$
connected components, since $\pi_1(PGL(2n,\mathbb{C})) = \mathbb{Z}/2n$. We write
\[
  X := \mathcal{M}_{\mathrm{reg}}(SL(2n,\mathbb{C})) = h^{-1}(\mathcal{A}_{\mathrm{reg}}), \:\:
  \hat X := \mathcal{M}^0_{\mathrm{reg}}(PGL(2n,\mathbb{C})) = \hat h^{-1}(\mathcal{A}_{\mathrm{reg}}) \cap \mathcal{M}^0(PGL(2n,\mathbb{C})),
\]
where $h$, $\hat h$ are the two Hitchin maps and $\mathcal{M}^0$ denotes the
component corresponding to $c = 0$. Finally,
$\mathcal{M}^{\mathrm{rev}}_{\mathrm{reg}} := \mathcal{M}^{\mathrm{rev}}(SL(2n,\mathbb{C})) \cap X$. By
Theorem~\ref{thm:BAA}, $\mathcal{M}^{\mathrm{rev}}_{\mathrm{reg}}$ is a $(B,A,A)$-brane, being
an open subset of $\mathcal{M}^{\mathrm{rev}}_*(SL(2n,\mathbb{C}))$.

% -------------------------------------------------------------------
\subsection{The fibrewise computation}\label{subsec:fibrewise}

Fix $a \in \mathcal{A}^{\mathrm{even}}_{\mathrm{reg}}$ and consider the abelian variety
\[
  A = \operatorname{Prym}(S_a/\Sigma) = \ker\big( \operatorname{Nm}_{S_a/\Sigma} \colon \operatorname{Pic}^0(S_a) \to \operatorname{Pic}^0(\Sigma) \big) \subset \operatorname{Pic}^0(S_a),
\]
which is biholomorphic to the regular fibre $h^{-1}(a)$ (see
\S\ref{subsec:FM}). The involution $\sigma$ acts on $A$ by the group
automorphism $L \mapsto \sigma^*L$, and we set
\[
  B := \{ L \in \operatorname{Prym}(S_a/\Sigma) : L \cong \sigma^*L \} \subset A,
  \qquad
  B_0 := p^* \operatorname{Prym}(C/\Sigma),
\]
where $\operatorname{Prym}(C/\Sigma) = \ker\big(\operatorname{Nm}_{C/\Sigma} \colon \operatorname{Pic}^0(C) \to \operatorname{Pic}^0(\Sigma)\big) \subset \operatorname{Pic}^0(C)$.
Let $\mathcal{F} \colon D^b(\mathrm{Coh}(A)) \to D^b(\mathrm{Coh}(A^\vee))$ denote
the Fourier--Mukai transform with kernel the normalized Poincar\'e line bundle
$\mathcal{L}_a$ on $A \times A^\vee$.

\begin{proposition}\label{prop:fibrewise}
With the notation above, the following hold.
\begin{enumerate}
  \item[(i)] $B_0 \subset A$, the map $p^* \colon \operatorname{Prym}(C/\Sigma) \to B_0$ is
    an isomorphism, and $B_0$ is the connected component of the identity of the
    algebraic group $B$. Consequently $B = B_0 \sqcup B_{L_1} \sqcup \cdots \sqcup B_{L_r}$
    is a finite disjoint union of cosets $B_{L_j} := L_j \otimes B_0$.
  \item[(ii)] The dual of the inclusion $B_0 \hookrightarrow A$ is the quotient map
    \[
      A^\vee = \frac{\operatorname{Pic}^0(S_a)}{\pi_a^* \operatorname{Pic}^0(\Sigma)} \twoheadrightarrow B_0^\vee = \frac{\operatorname{Pic}^0(C)}{\pi_C^* \operatorname{Pic}^0(\Sigma)}
    \]
    induced by $\operatorname{Nm}_{S_a/C}$.
  \item[(iii)] The kernel $B_0^\perp := \ker(A^\vee \twoheadrightarrow B_0^\vee)$ is
    \[
      B_0^\perp \cong \frac{\operatorname{Prym}(S_a/C)}{\operatorname{Prym}(S_a/C) \cap \pi_a^* \operatorname{Pic}^0(\Sigma)},
    \]
    the image of $\operatorname{Prym}(S_a/C)$ in $A^\vee$. It is a connected abelian
    subvariety of $A^\vee$.
  \item[(iv)] For the sheaf
    $\mathcal{E} := i_*\mathcal{O}_{B_0} \oplus i_*\mathcal{O}_{B_{L_1}} \oplus \cdots \oplus i_*\mathcal{O}_{B_{L_r}} = i_*\mathcal{O}_B$ on $A$,
    \[
      \operatorname{Supp} \mathcal{F}(\mathcal{E}) = B_0^\perp .
    \]
\end{enumerate}
\end{proposition}

\begin{proof}
(i) Let $L \in \operatorname{Prym}(C/\Sigma)$. Then
\[
  \operatorname{Nm}_{S_a/\Sigma}(p^*L) = \operatorname{Nm}_{C/\Sigma}\big(\operatorname{Nm}_{S_a/C}(p^*L)\big).
\]
Since $p$ is a double cover, $\operatorname{Nm}_{S_a/C}(p^*L) = L^{\otimes 2}$.
Substituting this gives
\[
  \operatorname{Nm}_{C/\Sigma}\big(L^{\otimes 2}\big) = \operatorname{Nm}_{C/\Sigma}(L)^{\otimes 2} = \mathcal{O}_\Sigma^{\otimes 2} = \mathcal{O}_\Sigma .
\]
This proves that $p^*\operatorname{Prym}(C/\Sigma) \subset A$. Moreover, since
$p \circ \sigma = p$, we have $\sigma^*p^*L \cong p^*L$, so $B_0 \subset B$.
Since $p$ is a ramified double cover, $p^* \colon \operatorname{Pic}^0(C) \to
\operatorname{Pic}^0(S_a)$ is injective \cite{Mum71}; therefore
$p^*\operatorname{Prym}(C/\Sigma) \cong \operatorname{Prym}(C/\Sigma)$.

The group $\operatorname{Prym}(C/\Sigma)$ is connected: if
$\pi_C^*M \cong \mathcal{O}_C$ for some $M \in \operatorname{Pic}^0(\Sigma)$, then
$M \otimes \pi_{C*}\mathcal{O}_C = \bigoplus_{i=0}^{n-1} M \otimes K^{-2i}$ has a
nonzero section, which forces $M \cong \mathcal{O}_\Sigma$; thus $\pi_C^*$ is
injective on $\operatorname{Pic}^0(\Sigma)$, and dually $\ker\operatorname{Nm}_{C/\Sigma}$
is connected. Hence $B_0$ is a connected subgroup of $B$ of dimension
$g(C) - g$.

Conversely, by Corollary~\ref{cor:identity-component}, the identity component
of $B$ is contained in $p^*\operatorname{Pic}^0(C) \cap A$. For $M \in
\operatorname{Pic}^0(C)$, the computation above shows that $p^*M \in A$ if and
only if $\operatorname{Nm}_{C/\Sigma}(M)^{\otimes 2} \cong \mathcal{O}_\Sigma$; the
identity component of this group is $\operatorname{Prym}(C/\Sigma)$. Hence $B_0$
is the identity component of $B$. Since $B$ is a closed subgroup of the compact
group $A$, it has finitely many connected components, and they are the cosets
of $B_0$.

(ii) The dual of the pullback map $p^* \colon \operatorname{Pic}^0(C) \to
\operatorname{Pic}^0(S_a)$ is the norm map
$\operatorname{Nm}_{S_a/C} \colon \operatorname{Pic}^0(S_a) \to \operatorname{Pic}^0(C)$
\cite{BL04}, and the duals of $A$ and $B_0 \cong \operatorname{Prym}(C/\Sigma)$ are
the quotients displayed in (ii). We show that $\operatorname{Nm}_{S_a/C}$ induces
a well-defined surjection between these quotients. Since the norm map from the
Jacobian of the cover to the Jacobian of the base is surjective, we have
\[
  \operatorname{Nm}_{S_a/C}\big(\operatorname{Pic}^0(S_a)\big) = \operatorname{Pic}^0(C).
\]
Now we evaluate the image of the subgroup $\pi_a^*\operatorname{Pic}^0(\Sigma)$
under $\operatorname{Nm}_{S_a/C}$. Let $L \in \operatorname{Pic}^0(\Sigma)$. Then
$\pi_a^*L = p^*(\pi_C^*L)$. Moreover, for any line bundle
$M \in \operatorname{Pic}(C)$ we have $\operatorname{Nm}_{S_a/C}(p^*M) = M^{\otimes 2}$.
Substituting $M = \pi_C^*L$, we get
\[
  \operatorname{Nm}_{S_a/C}\big(\pi_a^*L\big) = \big(\pi_C^*L\big)^{\otimes 2} = \pi_C^*\big(L^{\otimes 2}\big).
\]
Since $\operatorname{Pic}^0(\Sigma)$ is a complex torus, it is $2$-divisible,
i.e. squaring is surjective. Therefore
$\{L^{\otimes 2} : L \in \operatorname{Pic}^0(\Sigma)\} = \operatorname{Pic}^0(\Sigma)$, and
\[
  \operatorname{Nm}_{S_a/C}\big(\pi_a^*\operatorname{Pic}^0(\Sigma)\big) = \pi_C^*\operatorname{Pic}^0(\Sigma).
\]
This proves (ii).

(iii) By (ii),
\[
  B_0^\perp = \ker\left( \frac{\operatorname{Pic}^0(S_a)}{\pi_a^*\operatorname{Pic}^0(\Sigma)} \xrightarrow{\ \operatorname{Nm}_{S_a/C}\ } \frac{\operatorname{Pic}^0(C)}{\pi_C^*\operatorname{Pic}^0(\Sigma)} \right).
\]
Suppose a class $[L] \in A^\vee$ belongs to $B_0^\perp$. This means that
$\operatorname{Nm}_{S_a/C}(L) \in \pi_C^*\operatorname{Pic}^0(\Sigma)$, i.e.
$\operatorname{Nm}_{S_a/C}(L) = \pi_C^*N$ for some $N \in \operatorname{Pic}^0(\Sigma)$.
Since $\operatorname{Pic}^0(\Sigma)$ is $2$-divisible, we can choose
$M \in \operatorname{Pic}^0(\Sigma)$ with $M^{\otimes 2} = N$. Define
\[
  L' := L \otimes \pi_a^*(M^{-1}) \in \operatorname{Pic}^0(S_a).
\]
In particular $[L] = [L']$ in $A^\vee$. We compute the norm of this new
representative:
\[
  \operatorname{Nm}_{S_a/C}(L') = \operatorname{Nm}_{S_a/C}(L) \otimes \operatorname{Nm}_{S_a/C}\big(\pi_a^*(M^{-1})\big).
\]
Note that
\[
  \operatorname{Nm}_{S_a/C}\big(\pi_a^*(M^{-1})\big) = \operatorname{Nm}_{S_a/C}\big(p^*(\pi_C^*(M^{-1}))\big) = \pi_C^*\big(M^{-2}\big).
\]
Substituting these, we obtain
\[
  \operatorname{Nm}_{S_a/C}(L') = \pi_C^*N \otimes \pi_C^*\big(M^{-2}\big) = \pi_C^*\big(N \otimes M^{-2}\big) = \mathcal{O}_C,
\]
because $M^{\otimes 2} = N$. Therefore $L' \in \operatorname{Prym}(S_a/C)$. We have
shown that every class in $B_0^\perp$ admits a representative in
$\operatorname{Prym}(S_a/C)$. Conversely, for any $L \in \operatorname{Prym}(S_a/C)$, its
class in $A^\vee$ clearly belongs to $B_0^\perp$. Therefore $B_0^\perp$ is the
image of $\operatorname{Prym}(S_a/C)$ in $A^\vee$, which gives the displayed
isomorphism. Finally, $B_0^\perp$ is connected, since it is the dual
$(A/B_0)^\vee$ of an abelian variety (equivalently, since the kernel of the
norm of the ramified double cover $p$ is connected \cite{Mum71}).

(iv) By a theorem of Mukai \cite{Mu81}, for an abelian subvariety
$B_0 \subset A$, the transform $\mathcal{F}(i_*\mathcal{O}_{B_0})$ is, up to shift,
a sheaf supported exactly on $B_0^\perp$: for $\xi \in A^\vee$, the cohomology
$H^\bullet(B_0, \mathcal{L}_a|_{B_0 \times \{\xi\}})$ vanishes unless
$\mathcal{L}_a|_{B_0 \times \{\xi\}}$ is trivial, i.e. unless $\xi \in B_0^\perp$.
Hence $\operatorname{Supp} \mathcal{F}(i_*\mathcal{O}_{B_0}) = B_0^\perp$.

Now let $B_L = L \otimes B_0$ be another component of $B$, and consider the
translation $t_L \colon A \to A$, $t_L(L') = L' \otimes L$. Then
$t_{L*}(i_*\mathcal{O}_{B_0}) \cong i_*\mathcal{O}_{B_L}$. For any point $L \in A$, the
restriction of the Poincar\'e bundle gives a flat, degree-zero line bundle
$\mathcal{L}_L := \mathcal{L}_a|_{\{L\} \times A^\vee}$ on $A^\vee$, and translating a
sheaf on $A$ by $L$ corresponds to tensoring its transform by $\mathcal{L}_L$
\cite{Mu81}:
\[
  \mathcal{F}\big(t_{L*}(i_*\mathcal{O}_{B_0})\big) \cong \mathcal{F}(i_*\mathcal{O}_{B_0}) \otimes \mathcal{L}_L .
\]
Tensoring with a line bundle does not change the support, so
$\operatorname{Supp} \mathcal{F}(i_*\mathcal{O}_{B_L}) = B_0^\perp$. Since $\mathcal{F}$ is
additive and the direct sum defining $\mathcal{E}$ is finite,
\[
  \operatorname{Supp} \mathcal{F}(\mathcal{E}) = \operatorname{Supp} \mathcal{F}(i_*\mathcal{O}_{B_0}) \cup \operatorname{Supp} \mathcal{F}(i_*\mathcal{O}_{B_{L_1}}) \cup \cdots = B_0^\perp \cup B_0^\perp \cup \cdots = B_0^\perp . \qedhere
\]
This completes the proof. 
\end{proof}

% -------------------------------------------------------------------
\subsection{Global identifications and the relative Fourier--Mukai transform}\label{subsec:FM}

Throughout this section we work in the complex analytic category: the moduli spaces, the relative Prym abelian scheme $\boldsymbol{\mathcal A}$, its dual $\boldsymbol{\mathcal A}^\vee$ and the Poincaré bundle are algebraic, and we use their associated complex analytic spaces and analytic line bundles. Coherent sheaves are analytic coherent sheaves, and the results we use (flatness criteria, proper and flat base change, and the Fourier–Mukai results of \cite{Mu81}) hold in this setting.

\noindent\textbf{The $SL(2n,\mathbb{C})$ side.} By the relative BNR
correspondence \cite{bnr},
\[
  X = \mathcal{M}_{\mathrm{reg}}(SL(2n,\mathbb{C})) \cong \mathcal{N} := \operatorname{Nm}^{-1}\big(K^{n(2n-1)}\big) \subset \operatorname{Pic}^d(\mathcal{S}/\mathcal{A}_{\mathrm{reg}}), \qquad d = n(2n-1)(2g-2),
\]
where the norm condition comes from $\det \pi_{a*}L = \operatorname{Nm}(L) \otimes K^{-n(2n-1)}$
and $\det E \cong \mathcal{O}_\Sigma$. Fix a theta characteristic $K^{1/2}$ and
consider the line bundle $L_0 := \pi^*K^{(2n-1)/2}$ on $\mathcal{S}$. Since
$\operatorname{Nm}(L_0|_{S_a}) = K^{n(2n-1)}$, the map
\[
  \varphi \colon X \cong \mathcal{N} \longrightarrow \boldsymbol{\mathcal{A}} := \operatorname{Prym}(\mathcal{S}/\Sigma), \qquad L \longmapsto L \otimes L_0^{-1},
\]
is a global biholomorphism onto the relative Prym abelian scheme
$\boldsymbol{\mathcal{A}} \to \mathcal{A}_{\mathrm{reg}}$. Its restriction to the fibre over $a$ is the
biholomorphism
\[
  \varphi_a \colon h^{-1}(a) \cong \mathcal{N}_a \xrightarrow{\ \cong\ } A = \operatorname{Prym}(S_a/\Sigma), \qquad L \longmapsto L \otimes L_0(a)^{-1},
\]
where $\mathcal{N}_a := \{L \in \operatorname{Pic}^d(S_a) : \operatorname{Nm}(L) \cong K^{n(2n-1)}\}$
and $L_0(a) := L_0|_{S_a} = \pi_a^*K^{(2n-1)/2}$. The point
$\varphi_a^{-1}(\mathcal{O}_{S_a})$ is the Hitchin section \cite{Hi92}.

\begin{lemma}\label{lem:intertwine}
For $a \in \mathcal{A}^{\mathrm{even}}_{\mathrm{reg}}$, the biholomorphism $\varphi_a$ maps
$\mathcal{M}^{\mathrm{rev}}_{\mathrm{reg}} \cap h^{-1}(a)$ onto $B$.
\end{lemma}

\begin{proof}
By Theorem~\ref{thm:even-fibre}, under the BNR identification the involution
$(E,\Phi) \mapsto (E,-\Phi)$ of $h^{-1}(a)$ corresponds to $L \mapsto \sigma^*L$
(which preserves $\mathcal{N}_a$, as $\pi_a \circ \sigma = \pi_a$). Since
$\pi_a \circ \sigma = \pi_a$, we have
$\sigma^*L_0(a) = \sigma^*\pi_a^*K^{(2n-1)/2} \cong L_0(a)$, hence
\[
  \varphi_a(\sigma^*L) = \sigma^*L \otimes L_0(a)^{-1} = \sigma^*\big(L \otimes L_0(a)^{-1}\big) = \sigma^*\varphi_a(L).
\]
Thus $\varphi_a$ intertwines the two involutions, and therefore maps the fixed
locus $\mathcal{M}^{\mathrm{rev}}_{\mathrm{reg}} \cap h^{-1}(a)$ onto
$\operatorname{Fix}(\sigma^*) = B$.
\end{proof}

\noindent\textbf{The $PGL(2n,\mathbb{C})$ side.} Let $\pi^*\operatorname{Pic}^0(\Sigma)$
act on $\operatorname{Pic}^d(\mathcal{S}/\mathcal{A}_{\mathrm{reg}})$ by tensor product. A point of
$\hat X$ is represented by $(E,\Phi)$ with $\deg E = 0$, uniquely up to
$E \mapsto E \otimes M$ with $M \in \operatorname{Pic}^0(\Sigma)$; since
$\pi_*(L \otimes \pi^*M) = \pi_*L \otimes M$, the relative BNR correspondence
gives a canonical biholomorphism of $\hat X$ with the orbit space
$\operatorname{Pic}^d(\mathcal{S}/\mathcal{A}_{\mathrm{reg}})/\pi^*\operatorname{Pic}^0(\Sigma)$. Consider
\[
  \hat\varphi \colon \hat X \cong \operatorname{Pic}^d(\mathcal{S}/\mathcal{A}_{\mathrm{reg}})/\pi^*\operatorname{Pic}^0(\Sigma) \longrightarrow \boldsymbol{\mathcal{A}}^\vee := \operatorname{Pic}^0(\mathcal{S}/\mathcal{A}_{\mathrm{reg}})/\pi^*\operatorname{Pic}^0(\Sigma), \qquad [L] \longmapsto [L \otimes \pi^*N^{-1}],
\]
with $N$ a line bundle on $\Sigma$ of degree $(2n-1)(g-1)$. This is also a
global biholomorphism, and it is independent of the choice of $N$, since for
two such choices $N$ and $N'$, we have $\pi^*(N \otimes N'^{-1}) \in \pi^*\operatorname{Pic}^0(\Sigma)$.
In particular, since $\deg K^{(2n-1)/2} = (2n-1)(g-1)$,
we can take $N = K^{(2n-1)/2}$, which gives  $\hat\varphi_a([L]) = [L \otimes L_0(a)^{-1}]$.
Its fibrewise restriction is denoted by
\[
  \hat\varphi_a \colon \hat h^{-1}(a) \xrightarrow{\ \cong\ } A^\vee = \operatorname{Pic}^0(S_a)/\pi_a^*\operatorname{Pic}^0(\Sigma).
\]

\noindent\textbf{The Poincar\'e line bundle.} The relative Prym
$\boldsymbol{\mathcal{A}} \to \mathcal{A}_{\mathrm{reg}}$ is an abelian scheme, and
$\boldsymbol{\mathcal{A}}^\vee \to \mathcal{A}_{\mathrm{reg}}$ is its dual abelian scheme. Let
$\mathcal{P}^{\mathrm{univ}}$ be the normalized (universal) Poincar\'e line bundle on
$\boldsymbol{\mathcal{A}} \times_{\mathcal{A}_{\mathrm{reg}}} \boldsymbol{\mathcal{A}}^\vee$. Define the line bundle
\[
  \mathcal{P} := (\varphi \times \hat\varphi)^* \mathcal{P}^{\mathrm{univ}} \quad \text{on } X \times_{\mathcal{A}_{\mathrm{reg}}} \hat X .
\]
The restriction of $\mathcal{P}^{\mathrm{univ}}$ to each fibre $A \times A^\vee$ is the
classical normalized Poincar\'e bundle $\mathcal{L}_a$. In particular,
\begin{equation}\label{eq:3.1}
  \mathcal{P}_a := \mathcal{P}|_{h^{-1}(a) \times \hat h^{-1}(a)} = (\varphi_a \times \hat\varphi_a)^* \mathcal{L}_a .
\end{equation}

\noindent\textbf{The relative Fourier--Mukai transform.} Let
\[
  p_1 \colon X \times_{\mathcal{A}_{\mathrm{reg}}} \hat X \to X, \qquad p_2 \colon X \times_{\mathcal{A}_{\mathrm{reg}}} \hat X \to \hat X
\]
be the projections, and define
\[
  \Psi \colon D^b\big(\mathrm{Coh}(X)\big) \longrightarrow D^b\big(\mathrm{Coh}(\hat X)\big), \qquad
  \Psi(\mathcal{K}) = Rp_{2*}\big( Lp_1^*\mathcal{K} \otimes \mathcal{P} \big).
\]
No derived tensor product is needed, since $\mathcal{P}$ is a line bundle. For
$a \in \mathcal{A}_{\mathrm{reg}}$, let $j_a \colon h^{-1}(a) \hookrightarrow X$ and
$\hat{j}_a \colon \hat h^{-1}(a) \hookrightarrow \hat X$ be the inclusions, and let
$\Psi_a$ be the fibrewise Fourier--Mukai transform with kernel $\mathcal{P}_a$.

\begin{proposition}\label{prop:base-change}
\begin{enumerate}
  \item[(i)] $\Psi$ takes values in $D^b(\mathrm{Coh}(\hat X))$, and for every
    $\mathcal{K} \in D^b(\mathrm{Coh}(X))$ and $a \in \mathcal{A}_{\mathrm{reg}}$,
    \begin{equation}\label{eq:3.2}
      L\hat{j}_a^* \Psi(\mathcal{K}) \cong \Psi_a\big(Lj_a^*\mathcal{K}\big).
    \end{equation}
  \item[(ii)] For every $\mathcal{K}_a \in D^b(\mathrm{Coh}(h^{-1}(a)))$,
    $\Psi_a(\mathcal{K}_a) \cong \hat\varphi_a^*\, \mathcal{F}(\varphi_{a*}\mathcal{K}_a)$.
\end{enumerate}
\end{proposition}

\begin{proof}
(i) The map $p_2$ is proper, being a base change of the proper map
$h \colon X \to \mathcal{A}_{\mathrm{reg}}$; hence $Rp_{2*}$ preserves bounded coherent
complexes. The isomorphism \eqref{eq:3.2} follows from three facts:
derived base change for the proper map $p_2$ along the Cartesian square formed
by $\hat j_a$ and the inclusion of $h^{-1}(a) \times \hat h^{-1}(a)$, which is
Tor-independent because $p_2$ is flat (a base change of the flat map $h$);
functoriality of derived pullback along composites, which gives
$L(p_1|_{h^{-1}(a) \times \hat h^{-1}(a)})^* Lj_a^* \cong (\text{restriction}) \circ Lp_1^*$;
and the restriction \eqref{eq:3.1} of $\mathcal{P}$, cf.\ \cite{BBH09}.

(ii) By \eqref{eq:3.1}, the kernel of $\Psi_a$ is the pullback of the kernel of
$\mathcal{F}$ under the biholomorphism $\varphi_a \times \hat\varphi_a$; the
claim follows by a change of variables in the defining formula of the
Fourier--Mukai transform.
\end{proof}

We shall also use the following general fact about supports. Here the support
of a complex is the union of the supports of its cohomology sheaves.

\begin{lemma}\label{lem:support}
Let $f \colon Y \to Z$ be a holomorphic map of complex manifolds, and let
$\mathcal{K} \in D^b(\mathrm{Coh}(Z))$. Then
$\operatorname{Supp}(Lf^*\mathcal{K}) = f^{-1}\big(\operatorname{Supp}\mathcal{K}\big)$. In
particular, if $f$ is a closed immersion, then
$\operatorname{Supp}(Lf^*\mathcal{K}) = \operatorname{Supp}(\mathcal{K}) \cap Y$. Moreover, support
is invariant under isomorphisms in the derived category, finite direct sums and
shifts.
\end{lemma}

\begin{proof}
The inclusion $\subseteq$ holds since derived pullback is local. For
$\supseteq$, let $y \in f^{-1}(\operatorname{Supp}\mathcal{K})$, and let $q$ be the
largest integer with $\mathcal{H}^q(\mathcal{K})_{f(y)} \neq 0$. Since $Z$ is smooth,
$\mathcal{K}$ is locally quasi-isomorphic to a bounded complex of vector bundles,
and no cohomology above degree $q$ can contribute, so
$\mathcal{H}^q(Lf^*\mathcal{K})_y \cong \mathcal{H}^q(\mathcal{K})_{f(y)} \otimes_{\mathcal{O}_{Z,f(y)}} \mathcal{O}_{Y,y}$,
which is nonzero by Nakayama's lemma. The last assertion is immediate from the
definition, since isomorphic objects have isomorphic cohomology sheaves.
\end{proof}

\subsection{The structure sheaf of the reversible locus}\label{subsec:revsheaf}

Write $T := \mathcal{A}^{\mathrm{even}}_{\mathrm{reg}}$, and let
\[
  Y := h^{-1}(T) \subset X, \qquad \hat Y := \hat h^{-1}(T) \subset \hat X,
\]
with inclusions $\iota \colon Y \hookrightarrow X$ and
$\hat\iota \colon \hat Y \hookrightarrow \hat X$. On the regular locus all
fibres of $h$ are abelian varieties of dimension $\frac{1}{2}\dim X$; since $X$
is smooth and $\mathcal{A}_{\mathrm{reg}}$ is regular, $h$ is flat by miracle
flatness \cite[Theorem~23.1]{Mat89}, and having smooth fibres, it is a
submersion. The same holds for $\hat h$. Hence $Y$ and $\hat Y$ are closed
complex submanifolds of $X$ and $\hat X$, being preimages of the closed
submanifold $T \subset \mathcal{A}_{\mathrm{reg}}$ under submersions, and the
restrictions
\[
  h|_Y \colon Y \longrightarrow T, \qquad \hat h|_{\hat Y} \colon \hat Y \longrightarrow T
\]
are again submersions, hence flat. For $a \in T$, let
$y_a \colon h^{-1}(a) \hookrightarrow Y$ and
$\hat y_a \colon \hat h^{-1}(a) \hookrightarrow \hat Y$ be the inclusions, so
that $j_a = \iota \circ y_a$ and $\hat{j}_a = \hat\iota \circ \hat y_a$.

Every reversible Higgs bundle lies over $\mathcal{A}^{\mathrm{even}}$, and
$\mathcal{M}^{\mathrm{rev}}_{\mathrm{reg}}$ is closed in $X$, being a fixed
locus; hence we have a closed immersion
$k \colon \mathcal{M}^{\mathrm{rev}}_{\mathrm{reg}} \hookrightarrow Y$. Define
\[
  \mathcal{G} := k_*\mathcal{O}_{\mathcal{M}^{\mathrm{rev}}_{\mathrm{reg}}}, \qquad
  \mathcal{E}_{\mathrm{global}} := \iota_*\mathcal{G} = \iota_* k_* \mathcal{O}_{\mathcal{M}^{\mathrm{rev}}_{\mathrm{reg}}},
\]
regarded as an object of $D^b(\mathrm{Coh}(X))$ concentrated in degree $0$.
Pushforward along a closed immersion is exact, so no derived functor is needed
here, and $\operatorname{Supp}\mathcal{E}_{\mathrm{global}} = \mathcal{M}^{\mathrm{rev}}_{\mathrm{reg}}$.

\begin{proposition}\label{prop:revsheaf}
\begin{enumerate}
  \item[(i)] The restriction
    $p := h|_{\mathcal{M}^{\mathrm{rev}}_{\mathrm{reg}}} \colon \mathcal{M}^{\mathrm{rev}}_{\mathrm{reg}} \to T$
    is flat. Consequently $\mathcal{G}$ is flat over $T$ via $h|_Y$, and
    $Ly_a^*\mathcal{G} \cong y_a^*\mathcal{G}$ for every $a \in T$.
  \item[(ii)] For $a \in T$, the sheaf $\mathcal{H} := y_a^*\mathcal{G}$ is the
    structure sheaf of $\mathcal{M}^{\mathrm{rev}}_{\mathrm{reg}} \cap h^{-1}(a)$,
    pushed forward to $h^{-1}(a)$, and $\varphi_{a*}\mathcal{H} \cong i_*\mathcal{O}_B = \mathcal{E}$.
\end{enumerate}
\end{proposition}

\begin{proof}
(i) We apply miracle flatness \cite[Theorem~23.1]{Mat89}. The source
$\mathcal{M}^{\mathrm{rev}}_{\mathrm{reg}}$ is a complex manifold, being an open
subset of the complex submanifold $\operatorname{Fix}(\tilde\sigma)$ of
Theorem~\ref{thm:BAA}. The target $T$ is regular, being open in a vector space.
The map $p$ is holomorphic and surjective: for every $a \in T$,
Lemma~\ref{lem:intertwine} gives $p^{-1}(a) \cong B$, which is nonempty since
$\mathcal{O}_{S_a} \in B$ (as $\sigma^*\mathcal{O}_{S_a} \cong \mathcal{O}_{S_a}$).
By Lemma~\ref{lem:intertwine} and Proposition~\ref{prop:fibrewise}(i), every
fibre $p^{-1}(a) \cong B$ has dimension $\dim B_0 = g(C) - g$. This is constant
on $T$: by adjunction $g(S_a) = 4n^2(g-1)+1$, and the Riemann--Hurwitz formula
for $p \colon S_a \to C$, ramified at $4n(g-1)$ points, gives
$g(C) = n(2n-1)(g-1)+1$. Since $p$ is surjective with fibres of constant
dimension, this dimension equals
$\dim\mathcal{M}^{\mathrm{rev}}_{\mathrm{reg}} - \dim T$. Hence $p$ is flat.

Since $k$ is a closed immersion, at any point $y \in Y$ the stalk
$\mathcal{G}_y$ equals $\mathcal{O}_{\mathcal{M}^{\mathrm{rev}}_{\mathrm{reg}},y}$
if $y \in \mathcal{M}^{\mathrm{rev}}_{\mathrm{reg}}$, and $\mathcal{G}_y = 0$
otherwise. In particular each $\mathcal{G}_y$ is a module over
$\mathcal{O}_{T,h(y)}$ via the map of local rings induced by $h|_Y$. Since $p$
is flat, $\mathcal{O}_{\mathcal{M}^{\mathrm{rev}}_{\mathrm{reg}},y}$ is flat over
$\mathcal{O}_{T,h(y)}$ for all $y \in \mathcal{M}^{\mathrm{rev}}_{\mathrm{reg}}$;
at the points $y \in Y \setminus \mathcal{M}^{\mathrm{rev}}_{\mathrm{reg}}$,
$\mathcal{G}_y = 0$ is trivially flat. Therefore $\mathcal{G}$ is flat over $T$.
Since $h|_Y$ is also flat, Tor-independence (flat base change for derived
tensor products) shows that the derived and underived pullbacks of $\mathcal{G}$
to the fibre $h^{-1}(a)$ agree: $Ly_a^*\mathcal{G} \cong y_a^*\mathcal{G}$.

(ii) Write $M := \mathcal{M}^{\mathrm{rev}}_{\mathrm{reg}}$ and $F := h^{-1}(a)$;
both are closed complex submanifolds of $Y$. Let $I, J \subset \mathcal{O}_Y$ be
their ideal sheaves, i.e.\ the sheaves of holomorphic functions vanishing on
$M$ and on $F$, respectively. The set $Z := M \cap F$ is the common zero locus
of $I + J$, hence a closed analytic subset of $Y$; let $I_Z \subset \mathcal{O}_Y$
be the ideal sheaf of all holomorphic functions vanishing on $Z$. Clearly
$I + J \subset I_Z$. Let $k_a \colon Z \hookrightarrow F$ denote the inclusion.

\emph{Step 1: $\mathcal{H}$ is the structure sheaf of the intersection.} We
have the Cartesian square
\[
  \begin{array}{ccc}
    Z & \longrightarrow & M \\
    {\scriptstyle k_a}\downarrow & & \downarrow{\scriptstyle k} \\
    F & \xrightarrow{\ y_a\ } & Y
  \end{array}
\]
Since $k_*\mathcal{O}_M = \mathcal{O}_Y/I$, and restriction to the closed
submanifold $F$ is given by tensoring with $\mathcal{O}_F = \mathcal{O}_Y/J$, we
obtain
\[
  \mathcal{H} = y_a^*\big(k_*\mathcal{O}_M\big)
  \cong \big(\mathcal{O}_Y/I \otimes_{\mathcal{O}_Y} \mathcal{O}_Y/J\big)\big|_F
  \cong \big(\mathcal{O}_Y/(I+J)\big)\big|_F .
\]
Thus $\mathcal{H}$ is the structure sheaf of $Z$ regarded as the intersection of
the complex analytic subspaces $M$ and $F$. In general this can differ from the
sheaf $\mathcal{O}_Y/I_Z$ of holomorphic functions on $Z$, by nilpotent
elements, when $M$ and $F$ meet non-transversally. We now show that in our
situation $I + J = I_Z$.

\emph{Step 2: $I + J = I_Z$.} Let $\tilde\sigma(E,\Phi) = (E,-\Phi)$. It is a
holomorphic involution of $Y$, and $M$ is its fixed point set. We use the
following consequence of Cartan's linearization theorem: if $\tau$ is a
holomorphic involution of a complex manifold $W$, then near any fixed point
there are holomorphic coordinates $(u,v)$ on $W$ in which $\tau(u,v) = (u,-v)$.
In these coordinates $\operatorname{Fix}(\tau) = \{v = 0\}$, and its ideal
$(v_1,\dots,v_s)$ is generated by the functions $f - f\circ\tau$: indeed
$v_j - v_j\circ\tau = 2v_j$, and conversely every $f - f\circ\tau$ vanishes on
$\{v = 0\}$. Hence:
\begin{itemize}
  \item applied to $\tilde\sigma$ on $Y$, the ideal $I$ is generated by the
    functions $f - f\circ\tilde\sigma$, with $f$ local holomorphic functions on $Y$;
  \item since $a$ is even, $\tilde\sigma$ preserves $F$, and every local
    holomorphic function $g$ on $F$ is the restriction of one on $Y$. Therefore
    the image $(I+J)/J$ of $I$ in $\mathcal{O}_F = \mathcal{O}_Y/J$ is generated by
    the functions $g - g\circ\tilde\sigma|_F$;
  \item applied to $\tilde\sigma|_F$ on $F$, whose fixed point set is $Z$, these
    functions generate the ideal of all holomorphic functions on $F$ vanishing
    on $Z$, i.e.\ $I_Z/J$.
\end{itemize}
Hence $(I+J)/J = I_Z/J$, that is, $I + J = I_Z$.

\emph{Step 3: conclusion.} By Steps 1 and 2,
$\mathcal{H} \cong (\mathcal{O}_Y/I_Z)|_F = k_{a*}\mathcal{O}_Z$, where
$\mathcal{O}_Z$ is the sheaf of holomorphic functions on the complex submanifold
$Z = \mathcal{M}^{\mathrm{rev}}_{\mathrm{reg}} \cap h^{-1}(a)$. By
Lemma~\ref{lem:intertwine}, the biholomorphism $\varphi_a$ maps $Z$ onto $B$,
and pushforward along a biholomorphism carries the structure sheaf of a closed
submanifold to that of its image. Since $B$ is the finite disjoint union of the
cosets $B_{L_j}$, we conclude
\[
  \varphi_{a*}\mathcal{H} \cong i_*\mathcal{O}_B
  = i_*\mathcal{O}_{B_0} \oplus i_*\mathcal{O}_{B_{L_1}} \oplus \cdots = \mathcal{E}. \qedhere
\]
\end{proof}

% -------------------------------------------------------------------
\subsection{The fibrewise support of the transform}\label{subsec:support}

We compare $\Psi$ with the relative Fourier--Mukai transform over the even base
$T$. Let
\[
  W := X \times_{\mathcal{A}_{\mathrm{reg}}} \hat X, \qquad
  W_T := Y \times_T \hat Y,
\]
with closed immersion $\kappa \colon W_T \hookrightarrow W$, and let
$q_1 \colon W_T \to Y$ and $q_2 \colon W_T \to \hat Y$ be the projections. Then
$p_1 \circ \kappa = \iota \circ q_1$ and $p_2 \circ \kappa = \hat\iota \circ q_2$.
Set $\mathcal{P}_T := \kappa^*\mathcal{P}$, and define
\[
  \Psi_T \colon D^b\big(\mathrm{Coh}(Y)\big) \longrightarrow D^b\big(\mathrm{Coh}(\hat Y)\big), \qquad
  \Psi_T(\mathcal{K}) = Rq_{2*}\big( Lq_1^*\mathcal{K} \otimes \mathcal{P}_T \big).
\]
For $a \in T$, the restriction of $\mathcal{P}_T$ to
$h^{-1}(a) \times \hat h^{-1}(a)$ is $\mathcal{P}_a$, by \eqref{eq:3.1}.

\begin{lemma}\label{lem:evenbase}
\begin{enumerate}
  \item[(i)] For every $\mathcal{K} \in D^b(\mathrm{Coh}(Y))$,
    $\Psi(\iota_*\mathcal{K}) \cong \hat\iota_*\Psi_T(\mathcal{K})$.
  \item[(ii)] $\Psi_T$ takes values in $D^b(\mathrm{Coh}(\hat Y))$, and for every
    $\mathcal{K} \in D^b(\mathrm{Coh}(Y))$ and $a \in T$,
    \[
      L\hat y_a^*\,\Psi_T(\mathcal{K}) \cong \Psi_a\big(Ly_a^*\mathcal{K}\big).
    \]
\end{enumerate}
\end{lemma}

\begin{proof}
(i) Since $Y$ lies over $T$, a point $(x,\hat x)\in W$ with $x\in Y$ satisfies $\hat h(\hat x) = h(x)\in T$, so $\hat x\in\hat Y$. Hence $p_1^{-1}(Y) = Y\times_T\hat Y = W_T$, so the square formed by $\kappa$, $p_1$, $\iota$ and $q_1$ is Cartesian. The map $p_1$ is flat,
being a base change of the flat map $\hat h$. Hence flat base change gives
$Lp_1^*\,\iota_*\mathcal{K} \cong \kappa_*\,Lq_1^*\mathcal{K}$ (pushforward along
the closed immersions $\iota$ and $\kappa$ is exact). By the projection formula,
$\kappa_*Lq_1^*\mathcal{K} \otimes \mathcal{P} \cong \kappa_*\big(Lq_1^*\mathcal{K} \otimes \mathcal{P}_T\big)$.
Finally, $Rp_{2*}\kappa_* = R(p_2\circ\kappa)_* = R(\hat\iota\circ q_2)_* = \hat\iota_*\,Rq_{2*}$.
Combining these,
\[
  \Psi(\iota_*\mathcal{K}) = Rp_{2*}\big(Lp_1^*\iota_*\mathcal{K} \otimes \mathcal{P}\big)
  \cong \hat\iota_*\,Rq_{2*}\big(Lq_1^*\mathcal{K} \otimes \mathcal{P}_T\big) = \hat\iota_*\Psi_T(\mathcal{K}).
\]

(ii) The map $q_2$ is the base change of $h|_Y \colon Y \to T$ along $\hat h|_{\hat Y}$;
since $h|_Y$ is proper and flat, so is $q_2$. As $q_1$ is flat and $q_2$ is proper,
$\Psi_T$ takes values in $D^b(\mathrm{Coh}(\hat Y))$ by Grauert's direct image theorem.

Fix $a \in T$ and put $W_a := h^{-1}(a) \times \hat h^{-1}(a)$, with inclusion
$\kappa_a \colon W_a \hookrightarrow W_T$ and projections $r_1, r_2$ onto the two factors.
Then $W_a = q_2^{-1}\big(\hat h^{-1}(a)\big)$, $q_1 \circ \kappa_a = y_a \circ r_1$, and
$\kappa_a^*\mathcal{P}_T = \mathcal{P}_a$ by \eqref{eq:3.1}. Since $q_2$ is proper and flat,
base change gives
\[
  L\hat y_a^*\, Rq_{2*} \cong Rr_{2*}\, L\kappa_a^* .
\]
Hence
\[
  L\hat y_a^*\,\Psi_T(\mathcal{K})
  \cong Rr_{2*}\big( L\kappa_a^* Lq_1^*\mathcal{K} \otimes \kappa_a^*\mathcal{P}_T \big)
  \cong Rr_{2*}\big( Lr_1^*\, Ly_a^*\mathcal{K} \otimes \mathcal{P}_a \big)
  = \Psi_a\big(Ly_a^*\mathcal{K}\big).
\]
\end{proof}

\begin{proposition}\label{prop:support}
For every $a \in \mathcal{A}_{\mathrm{reg}}$,
\begin{equation}\label{eq:3.4}
  \operatorname{Supp}\Psi(\mathcal{E}_{\mathrm{global}}) \cap \hat h^{-1}(a)
  = \operatorname{Supp} L\hat{j}_a^*\Psi(\mathcal{E}_{\mathrm{global}}) =
  \begin{cases}
    \hat\varphi_a^{-1}\big(B_0^\perp\big), & a \in \mathcal{A}^{\mathrm{even}}_{\mathrm{reg}},\\[2pt]
    \emptyset, & a \in \mathcal{A}_{\mathrm{reg}} \setminus \mathcal{A}^{\mathrm{even}}_{\mathrm{reg}}.
  \end{cases}
\end{equation}
\end{proposition}

\begin{proof}
The first equality follows from Lemma~\ref{lem:support}, since
$\hat{j}_a \colon \hat h^{-1}(a) \hookrightarrow \hat X$ is a closed immersion of
complex manifolds and $\Psi(\mathcal{E}_{\mathrm{global}}) \in D^b(\mathrm{Coh}(\hat X))$
by Proposition~\ref{prop:base-change}(i).

By Lemma~\ref{lem:evenbase}(i),
$\Psi(\mathcal{E}_{\mathrm{global}}) = \Psi(\iota_*\mathcal{G}) \cong \hat\iota_*\Psi_T(\mathcal{G})$.
Since $\hat\iota_*$ is exact and preserves supports,
\[
  \operatorname{Supp}\Psi(\mathcal{E}_{\mathrm{global}}) =\operatorname{Supp}\Psi_T(\mathcal{G}) \subset \hat Y = \hat h^{-1}(T).
\]
If $a \in \mathcal{A}_{\mathrm{reg}} \setminus T$, then $\hat h^{-1}(a)$ is
disjoint from $\hat Y$, so the intersection in \eqref{eq:3.4} is empty.

Now let $a\in T$. Since $\Psi_T(\mathcal G)\in D^b(\mathrm{Coh}(\hat Y))$ by Lemma~\ref{lem:evenbase}(ii), applying Lemma~\ref{lem:support}  to the closed immersion $\hat y_a:\hat h^{-1}(a)\hookrightarrow\hat Y$, we get 
$$\mathrm{Supp}\,\Psi(\mathcal E_{\mathrm{global}})\cap\hat h^{-1}(a) = \mathrm{Supp}\,\Psi_T(\mathcal G)\cap\hat h^{-1}(a) = \mathrm{Supp}\,L\hat y_a^*\,\Psi_T(\mathcal G).$$

By Lemma~\ref{lem:evenbase}(ii) and Proposition~\ref{prop:revsheaf}(i),
\[
  L\hat y_a^*\,\Psi_T(\mathcal{G}) \cong \Psi_a\big(Ly_a^*\mathcal{G}\big) \cong \Psi_a(\mathcal{H}),
\]
and by Proposition~\ref{prop:base-change}(ii) and
Proposition~\ref{prop:revsheaf}(ii),
\[
  \Psi_a(\mathcal{H}) \cong \hat\varphi_a^*\,\mathcal{F}(\varphi_{a*}\mathcal{H}) \cong \hat\varphi_a^*\,\mathcal{F}(\mathcal{E}).
\]
Since $\hat\varphi_a$ is a biholomorphism, Proposition~\ref{prop:fibrewise}(iv)
gives
$\operatorname{Supp}\hat\varphi_a^*\,\mathcal{F}(\mathcal{E}) = \hat\varphi_a^{-1}\big(B_0^\perp\big)$.
This proves \eqref{eq:3.4}.
\end{proof}

% -------------------------------------------------------------------
\subsection{The \texorpdfstring{$PSp(2n,\mathbb{C})$}{PSp(2n,C)}-locus}\label{subsec:psp}

The moduli space of $PSp(2n,\mathbb{C})$-Higgs bundles has two connected
components, indexed by $\pi_1(PSp(2n,\mathbb{C})) = \mathbb{Z}/2 = \{0,1\}$. Extension
of structure group along $PSp(2n,\mathbb{C}) \subset PGL(2n,\mathbb{C})$ defines a map
\[
  e \colon \mathcal{M}^0_{\mathrm{reg}}(PSp(2n,\mathbb{C})) \longrightarrow \hat X, \qquad (E,\Phi) \longmapsto \big(E \times_{PSp(2n,\mathbb{C})} PGL(2n,\mathbb{C}),\ \Phi\big),
\]
where $\Phi\in H^0(E(\mathfrak{psp}(2n,\mathbb C))\otimes K)$ is regarded as a section of $E'(\mathfrak{pgl}(2n,\mathbb C))\otimes K$, $E' := E\times_{PSp(2n,\mathbb C)}PGL(2n,\mathbb C)$, via the $\mathrm{Ad}$-equivariant inclusion $\mathfrak{psp}(2n,\mathbb C) = \mathfrak{sp}(2n,\mathbb C)\subset\mathfrak{sl}(2n,\mathbb C) = \mathfrak{pgl}(2n,\mathbb C)$, which makes $E(\mathfrak{psp})$ a subbundle of $E'(\mathfrak{pgl})$. Here
$\mathcal{M}^0_{\mathrm{reg}}(PSp(2n,\mathbb{C}))$ denotes the points of the degree-$0$
component whose spectral curve is smooth. The map lands in the degree-$0$
component of $\mathcal{M}(PGL(2n,\mathbb{C}))$, since the induced map
$\mathbb{Z}/2 \to \mathbb{Z}/2n$ on fundamental groups sends $0$ to $0$ (it sends the
class of $-I$ to $n$).

\begin{proposition}\label{prop:psp}
The map $e$ is a closed embedding whose image lies over
$\mathcal{A}^{\mathrm{even}}_{\mathrm{reg}}$, and for every $a \in \mathcal{A}^{\mathrm{even}}_{\mathrm{reg}}$,
\[
  \hat\varphi_a\Big( e\big(\mathcal{M}^0_{\mathrm{reg}}(PSp(2n,\mathbb{C}))\big) \cap \hat h^{-1}(a) \Big) = B_0^\perp .
\]
\end{proposition}

\begin{proof}
The Hitchin base of $PSp(2n,\mathbb{C})$ is
$\bigoplus_{i=1}^n H^0(\Sigma, K^{2i}) = \mathcal{A}^{\mathrm{even}}$, since the odd invariant
polynomials vanish on $\mathfrak{psp}(2n,\mathbb{C})=\mathfrak{sp}(2n,\mathbb{C})$, and the spectral curve of a point
$a$ is the same curve $S_a$. Hence $\hat h \circ e$ is the composition of the
$PSp$-Hitchin map with the inclusion $\mathcal{A}^{\mathrm{even}} \subset \mathcal{A}$, and the
image of $e$ lies over $\mathcal{A}^{\mathrm{even}}_{\mathrm{reg}}$.

Fix $a \in \mathcal{A}^{\mathrm{even}}_{\mathrm{reg}}$. An $Sp(2n,\mathbb{C})$-Higgs bundle with spectral
curve $S_a$ corresponds to a line bundle $L$ on $S_a$ with
$L \otimes \sigma^*L \cong \pi_a^*K^{2n-1}$ \cite{Hi07}. Writing
$M = L \otimes L_0(a)^{-1}$ and using $\sigma^*L_0(a) \cong L_0(a)$, this becomes
$M \otimes \sigma^*M \cong \mathcal{O}_{S_a}$, i.e. $p^*\operatorname{Nm}_{S_a/C}(M) \cong \mathcal{O}$,
i.e. $M \in \operatorname{Prym}(S_a/C)$, by injectivity of $p^*$. The degree-$0$
component of the $PSp(2n,\mathbb{C})$-fibre is the quotient of the
$Sp(2n,\mathbb{C})$-fibre by the action of $\operatorname{Jac}(\Sigma)[2] = H^1(\Sigma,\mathbb{Z}/2)$,
$E \mapsto E \otimes M'$, which on spectral data is $L \mapsto L \otimes \pi_a^*M'$.
Hence the $PSp(2n,\mathbb{C})$-fibre is identified with
$\operatorname{Prym}(S_a/C)/\pi_a^*\operatorname{Jac}(\Sigma)[2]$, and, since $\hat\varphi_a$ is
also given by $[L] \mapsto [L \otimes L_0(a)^{-1}]$, the map $e$ on this fibre is
\[
  \operatorname{Prym}(S_a/C)/\pi_a^*\operatorname{Jac}(\Sigma)[2] \longrightarrow A^\vee = \operatorname{Pic}^0(S_a)/\pi_a^*\operatorname{Pic}^0(\Sigma), \qquad [M] \longmapsto [M].
\]
Its image is the image of $\operatorname{Prym}(S_a/C)$ in $A^\vee$, which is $B_0^\perp$
by Proposition~\ref{prop:fibrewise}(iii). It is injective: if
$\pi_a^*M' \in \operatorname{Prym}(S_a/C)$ with $M' \in \operatorname{Pic}^0(\Sigma)$, then
$\pi_C^*(M'^{\otimes 2}) = \operatorname{Nm}_{S_a/C}(\pi_a^*M') = \mathcal{O}_C$, so
$M'^{\otimes 2} \cong \mathcal{O}_\Sigma$ by injectivity of $\pi_C^*$ (proof of
Proposition~\ref{prop:fibrewise}(i)); thus
$\operatorname{Prym}(S_a/C) \cap \pi_a^*\operatorname{Pic}^0(\Sigma) = \pi_a^*\operatorname{Jac}(\Sigma)[2]$.

Consequently $e$ is injective, since it is injective on fibres and compatible
with the injective map of bases. It is a holomorphic immersion, its
differential being injective along the fibres (inclusion of an abelian
subvariety) and along the base (the linear inclusion
$\mathcal{A}^{\mathrm{even}} \subset \mathcal{A}$). Finally, $e$ is proper, since both
$\mathcal{M}^0_{\mathrm{reg}}(PSp(2n,\mathbb{C}))$ and $\hat h^{-1}(\mathcal{A}^{\mathrm{even}}_{\mathrm{reg}}) \cap \hat X$ are
proper over $\mathcal{A}^{\mathrm{even}}_{\mathrm{reg}}$. An injective proper immersion is a closed
embedding.
\end{proof}

% -------------------------------------------------------------------
\subsection{Proof of Theorem~\ref{thm:main}}\label{subsec:proof}

By Proposition~\ref{prop:support}, the support of $\Psi(\mathcal{E}_{\mathrm{global}})$
 lies entirely inside the even regular
Hitchin fibres, and its intersection with every even fibre $\hat h^{-1}(a)$ is
$\hat\varphi_a^{-1}(B_0^\perp)$. By Proposition~\ref{prop:psp}, the subset $e\big(\mathcal{M}^0_{\mathrm{reg}}(PSp(2n,\mathbb{C}))\big)$ has exactly the same
intersection with every fibre. Since both sets are unions of their
intersections with the fibres of $\hat h$, we conclude that
\[
  \operatorname{Supp}\Psi(\mathcal{E}_{\mathrm{global}}) = e\big(\mathcal{M}^0_{\mathrm{reg}}(PSp(2n,\mathbb{C}))\big).
\]

It remains to show that this support is of type $(B,B,B)$. Note that
$\mathcal{M}^0_{\mathrm{reg}}(PSp(2n,\mathbb{C}))$ is itself a hyper-K\"ahler manifold, since it
is an open subset of the smooth locus of $\mathcal{M}(PSp(2n,\mathbb{C}))$, which carries
the hyper-K\"ahler structure of \S\ref{revbrane}.  Denote its complex structures by
$(J_1', J_2', J_3')$. Similarly, $\hat X$ is a hyper-K\"ahler manifold; denote
its complex structures by $(J_1, J_2, J_3)$. The map $e$ is holomorphic with
respect to the three pairs $(J_i', J_i)$, $i = 1,2,3$:
\begin{itemize}
  \item for $i = 1$, extension of structure group takes holomorphic families of Higgs bundles to holomorphic families, so $e$ is holomorphic with respect to $J_1$;
  \item for $i = 2$, under the non-abelian Hodge correspondence $e$ corresponds
    to the map of character varieties
    $\mathcal{X}_{PSp(2n,\mathbb{C})}(\Gamma) \to \mathcal{X}_{PGL(2n,\mathbb{C})}(\Gamma)$ induced by
    $PSp(2n,\mathbb{C}) \subset PGL(2n,\mathbb{C})$, since a harmonic reduction to the
    maximal compact subgroup of $PSp(2n,\mathbb{C})$ extends to one for
    $PGL(2n,\mathbb{C})$; this map is algebraic, hence $J_2$-holomorphic;
  \item for $i = 3$, $de \circ J_3' = de \circ J_1'J_2' = J_1 \circ de \circ J_2' = J_1 J_2 \circ de = J_3 \circ de$.
\end{itemize}
Since $e$ is a closed embedding which is holomorphic with respect to each pair
$(J_i', J_i)$, its image
$e\big(\mathcal{M}^0_{\mathrm{reg}}(PSp(2n,\mathbb{C}))\big)$ is a complex submanifold of $\hat X$,
and hence of the smooth locus of $\mathcal{M}(PGL(2n,\mathbb{C}))$, with respect to each of
$J_1, J_2, J_3$. Hence $\operatorname{Supp}\Psi(\mathcal{E}_{\mathrm{global}})$ is of type
$(B,B,B)$. \qed

\begin{remark}
In the framework of homological mirror symmetry, a $B$-brane is an object of
the bounded derived category of coherent sheaves, and its support encodes
where the brane lives. Throughout the proof we regard the coherent sheaf
$\mathcal{E}_{\mathrm{global}}$ as an object of $D^b(\mathrm{Coh}(\mathcal{M}_{\mathrm{reg}}(SL(2n,\mathbb{C}))))$;
its support is precisely $\mathcal{M}^{\mathrm{rev}}_{\mathrm{reg}}(SL(2n,\mathbb{C}))$. Applying the
relative Fourier--Mukai transform $\Psi$, we obtain
$\Psi(\mathcal{E}_{\mathrm{global}}) \in D^b(\mathrm{Coh}(\mathcal{M}^0_{\mathrm{reg}}(PGL(2n,\mathbb{C}))))$,
whose support, biholomorphic to $\mathcal M^0_{\mathrm{reg}}(PSp(2n,\mathbb C))$, is of type $(B,B,B)$; in this sense it is the mirror of the $(B,A,A)$-brane $\mathcal M^{\mathrm{rev}}_{\mathrm{reg}}(SL(2n,\mathbb C))$. This support does not depend on the choice of the theta characteristic used to define $\varphi$. Indeed, another choice replaces $\varphi$ by its composition with a translation of $\boldsymbol{\mathcal A}$. This replaces $\mathcal P$ by $\mathcal P\otimes p_2^*\mathcal N$ for some line bundle $\mathcal N$ on $\hat X$, and hence, by the projection formula, $\Psi(\mathcal E_{\mathrm{global}})$ by $\Psi(\mathcal E_{\mathrm{global}})\otimes\mathcal N$, which has the same support.
\end{remark}

\subsection*{Acknowledgements} 
The authors thank Indranil Biswas, Sibasish Banerjee, and Pradip Kumar for discussions on topics related to this work. We are also grateful to Szilárd Szabó and Oscar García-Prada for their comments on an earlier draft and for drawing our attention to \cite{Si09,hh22} and \cite{ga,gr}, respectively.

Part of this work was carried out during Gongopadhyay’s visit to the Institut des Hautes Études Scientifiques and Nayak’s visits to ICTS Bengaluru (ICTS/GSS2026/02) and the Lodha Mathematical Sciences Institute, Mumbai. They thank these institutes for their hospitality and support.

\section*{Declarations} 

{\bf Ethical Approval}. Not applicable.

{\bf Competing Interest}. Not applicable.

{\bf Funding}. Gongopadhyay acknowledges support from ANRF project: \\ ANRF/ARGM/2025/000122/MTR.

{\bf Authors' Contributions.} All authors have contributed equally. 

{\bf Availability of data and materials}. Not applicable.

\medskip {\bf AI declaration.}
During the preparation of this manuscript, the authors used LLM tools for language editing, stylistic improvement, and proofreading, including assistance in identifying possible inconsistencies in the exposition. The mathematical results and proofs were developed and independently verified by the authors. The authors reviewed the manuscript and assume full responsibility for its content.

\end{document}